\documentclass{article}

\usepackage[main, final]{neurips_2026}

\usepackage{amsfonts}       
\usepackage{nicefrac}       
\usepackage{microtype}      
\usepackage{xcolor}         
\usepackage{amsmath}
\usepackage{amssymb}
\usepackage{mathtools}
\usepackage{amsthm}
\usepackage[utf8]{inputenc} 
\usepackage[T1]{fontenc}    
\usepackage{hyperref}       
\usepackage{url}            
\usepackage{booktabs}       
\usepackage{nicefrac}       
\usepackage{microtype}      
\usepackage{xcolor}         
\usepackage{float}
\usepackage{tikz}
\usepackage{algorithm}
\usepackage{algorithmic}

\newtheorem{lemma}{Lemma}
\newtheorem{corollary}{Corollary}
\newtheorem{theorem}{Theorem}
\newtheorem{assump}{Assumption}
\newtheorem{fact}{Fact}

\newtheorem{definition}{Definition}

\newcommand{\real}{\mathbb{R}}

\newcommand{\expec}{\mathbb{E}}
\newcommand{\cA}{\mathcal{A}}

\newcommand{\cC}{\mathcal{C}}

\newcommand{\cS}{\mathcal{S}}
\newcommand{\cF}{\mathcal{F}}
\newcommand{\cM}{\mathcal{M}}

\newcommand{\cP}{P}
\newcommand{\cR}{\mathcal{R}}

\newcommand{\cT}{\mathcal{T}}
\newcommand{\cX}{\mathcal{X}}

\newcommand*{\infn}[1]{\left\|{#1}\right\|_{\infty}}

\DeclareMathOperator{\dom}{dom}
\DeclareMathOperator{\rint}{rint}

\DeclareMathOperator*{\argmin}{argmin}
\DeclareMathOperator*{\argmax}{argmax}
\title{Policy Gradient Algorithms in\\ Average-Reward Multichain MDPs}

\author{Jongmin Lee\\
University of Pennsylvania\\
Department of Electrical and Systems Engineering\\
\texttt{jongmlee@engineering.upenn.edu} \\
\And
Ernest K. Ryu \\
UCLA \\
Department of Mathematics \\
\texttt{eryu@math.ucla.edu} 
}

\begin{document}

\maketitle

\begin{abstract}%
While there is an extensive body of research analyzing policy gradient methods for discounted cumulative-reward MDPs, prior work on policy gradient methods for average-reward MDPs has been limited, with most existing results restricted to ergodic or unichain settings. In this work, we establish a policy gradient theorem formulated in terms of recurrent and transient visitation measures for average-reward multichain MDPs based on the invariance of the classification of recurrent and transient states. Building on this foundation, we develop refined analyses and obtain a collection of convergence and sample-complexity results that advance the understanding of this setting. In particular, the proposed $\alpha$-clipped policy mirror ascent algorithm, the first model-free method for multichain MDPs, attains an $\epsilon$-optimal policy with respect to positive policies.
\end{abstract}

\section{Introduction}

Average-reward Markov decision processes (MDPs) provide a fundamental framework for modeling sequential decision-making problems in which the objective is to maximize long-term, steady-state performance.
In the dynamic programming and reinforcement learning literature, both value-based and policy-based methods have been extensively studied in the average-reward setting. Since the seminal Policy Gradient Theorem of \cite{sutton1999}, performance-gradient formulas have also been established for multichain MDPs \citep{cao2004unified}. Existing convergence analyses of policy gradient methods under the average-reward criterion are restricted to ergodic or unichain MDPs, leaving the convergence of policy gradient methods in multichain MDPs open.




\paragraph{Contribution.}
In this work, we study the convergence of policy gradient methods for average-reward multichain MDPs. Using the notion of recurrent-transient classification of states, we establish a policy gradient theorem for multichain MDPs in terms of recurrent and transient visitation measures: 

\textbf{Theorem~\ref{thm::policy_gd}} (Recurrent and transient policy gradient). For $\pi_{\theta} \in \Pi_{+}$  and $\mu\in \mathcal{M}_+(\mathcal{S})$,
\[\nabla_{\theta}  J_\mu^{\pi_\theta}
=    \sum_{s \in \cR}
\sum_{a \in \cA}
d^{\pi_\theta}_{\mu}(s) \nabla_{\theta} \pi_{\theta} (a \mid s) Q^{\pi_\theta}(s,a) 
 +\sum_{s \in \cT}
\sum_{a \in \cA}  \delta^{\pi_\theta}_{\mu}(s)\nabla_{\theta} \pi_{\theta} (a \mid s) PJ^{\pi_\theta}(s,a),
\]
where $\cR,\cT$ are sets of recurrent and transient states, and  $d_\mu^{\pi_\theta}, \delta^{\pi_\theta}_{\mu}$ are \emph{recurrent and transient visitation measures}, respectively.
Building on this framework, we analyze the convergence of $\alpha$-clipped policy mirror ascent in the tabular setting and further extend the analysis to the generative model setting. Notably, we establish $\epsilon$-optimality guarantees with respect to positive policies for the general multichain MDPs, in both the tabular and generative model settings. To the best of our knowledge, our stochastic $\alpha$-clipped policy mirror ascent is the first model-free method for multichain average-reward MDPs. 

\subsection{Prior works}
\paragraph{Average-reward MDPs.}
The setup of average-reward MDPs was first introduced in the dynamic programming literature by \cite{howard1960dynamic}, and \cite{blackwell1962discrete} established a theoretical framework for their analysis. In reinforcement learning (RL), average-reward MDPs were mainly considered in the sample-based setup where the transition matrix and reward are unknown \citep{mahadevan1996average}. For this setup, various methods were proposed: model-based methods \citep{jin2021towards, tuynman2024finding, zurek2024span}, value-based methods \citep{ wan2024convergence, bravo2024stochastic, chen2025non}, and policy gradient methods \citep{bai2024regret,murthy2023convergence, kumar2024global}. Specifically, sample complexity under a generative model \citep{wang2017primal, zhang2023sharper, li2025stochastic, jin2024feasible, lee2025near} and with an offline dataset \citep{ozdaglar2024, gabbianelli2024offline, zurek2025optimal, lee2025finite}, and regret minimization frameworks \citep{burnetas1997optimal, Jaksch2010, zhang2019regret, boone2024achieving} have been actively studied on the theoretical side. On the practical side, average-reward RL has been applied to continual tasks such as access-control queuing \cite{sutton2018reinforcement}, inventory and supply-chain management \cite{geevers2024multi}, traffic signal control \cite{prashanth2010reinforcement}, robotics \cite{mahadevan1996average}, linear temporal logic \cite{le2024reinforcement} and other long-term decision-making problems.

\paragraph{Policy gradient methods.}
Policy gradient methods \citep{williams1992simple,sutton1999,konda1999actor,kakade2001natural} are foundational reinforcement learning algorithms, commonly implemented with deep neural networks for policy parameterization \citep{schulman2015trust,schulman2017proximal}. In line with their practical success, convergence and sample complexity of policy gradient variants have been extensively studied across settings \citep{shani2020adaptive, mei2020global, agarwal2021theory, cen2022fast, xiao2022convergence, bhandari2024global}. 

For the average-reward MDP, \cite{sutton1999, marbach2001simulation, baxter2000direct} establish a policy gradient theorem for unichain MDPs. More recently, \cite{murthy2023convergence} establishes global convergence of natural policy gradient under an irreducibility assumption on MDP, and \cite{kumar2024global} provides a  global convergence analysis of projected policy gradient in the tabular ergodic MDP. Regret guarantees for general parameterized policy gradient methods were studied by \cite{bai2024regret} and subsequently, \cite{ganesh2024order} improved sample complexity via variance reduction technique. In the unichain MDP, \cite{ganesh2025regret} analyzes a batched natural actor--critic algorithm and \cite{li2025stochastic} establishes sample-complexity guarantees for both generative and Markovian sampling models assuming mixing time.

In multichain and weakly communicating MDPs, performance-gradient formulas exist \citep{cao2004unified}. To the best of our knowledge, convergence guarantees for policy gradient methods have not been established, whereas model-based and value-based methods exist \citep{wei2020model,  zurek2024span, lee2025optimal}.

\subsection{Preliminaries and notations}

\paragraph{Average-reward MDP.}
Let $\cM(\cX)$ be the space of probability distributions over a set $\cX$. Write $(\cS, \cA, P, r)$ to denote the infinite-horizon undiscounted MDP with finite state space $\cS$, finite action space $\cA$, transition matrix $P\colon \cS \times \cA \rightarrow \cM(\cS)$, and bounded reward $r\colon  \cS \times \cA \rightarrow [-R,R]$. Denote $\pi\colon \cS \rightarrow \cM(\cA)$ for a policy, 
\begin{align*}
    &J^{\pi}(s)=\lim_{H\rightarrow \infty} \frac{1}{H}\expec_{\pi}\bigg[\sum^{H-1}_{h=0}  r(s_h, a_h) \,\Big|\, s_0=s\bigg]
    ,\\ &PJ^{\pi}(s,a)=\lim_{H\rightarrow \infty} \frac{1}{H}\expec_{\pi}\bigg[\sum^{H-1}_{h=0}  r(s_h, a_h) \,\Big|\, s_0=s, a_0=a\bigg],
\end{align*}
for average-rewards of a given policy \cite[Section 9.1]{10.5555/528623}, and \begin{align*}
&V^\pi(s)=\lim _{H\rightarrow \infty} \frac{1}{H}\sum_{h=1}^{H} \expec_\pi\bigg[\sum^{h-1}_{i=0}\big(r^\pi(s_i)-J^{\pi}(s_i)\big) \,\Big|\, s_0=s\bigg]
\\&Q^\pi(s,a)=\lim _{H\rightarrow \infty} \frac{1}{H}\sum_{h=1}^{H} \expec_\pi \bigg[\sum^{h-1}_{i=0}\big(r(s_i,a_i)-PJ^{\pi}(s_i,a_i)\big)\,\Big|\, s_0=s, a_0=a\bigg]    
\end{align*}
for state and state-action relative value (under an aperiodicity assumption, averaging with respect to $H$ can be omitted in the definition \cite[Section 8.2]{10.5555/528623}), where $\expec_{\pi}$ denotes the expected value over all trajectories $(s_0, a_0, s_1, a_1, \dots, s_{H-1}, a_{H-1})$ induced by $P$ and $\pi$, and $r^{\pi}(s)=\mathbb{E}_{a \sim \pi(\cdot\,|\,s) }\left[r(s,a)\right]$ as the reward induced by policy $\pi$. Given $\mu \in \cM(\cS)$,  define $J_\mu^\pi= \expec_{s_0 \sim \mu} [J^\pi_{s_0}]$.  We say $J^{\star}$ is optimal average reward if $J^{\star}=\max_{\pi}J^{\pi}$ and optimal average reward always exists in multichain MDP \citep[Section 9.1]{10.5555/528623}. Denote  $
    \cP^{\pi}(s , s')=
\mathrm{Prob}(s\rightarrow s'\,|\,
a \sim \pi(\cdot\,|\,s), s'\sim P(\cdot\,|\,s,a))$
as the transition matrix induced by policy $\pi$.
We denote 
\[
P_{\star}= 
\lim_{H\rightarrow \infty}\frac{1}{H} \sum^{H-1}_{i=0}P^i
\]
for the Ces\`aro limit of a stochastic matrix $P$. (Ces\`aro limit of stochastic matrix always exists \cite[Theorem A.6]{10.5555/528623}). Then, by definition, we can write  \cite[Section 8.2]{10.5555/528623}
\[J^\pi=P^\pi_{\star}r^\pi, \qquad V^\pi=(I-P^\pi+P^{\pi}_\star)^{-1}(I-P_\star^\pi)r^\pi,
\qquad
Q^\pi=r+P V^\pi-PJ^\pi.
\]
Lastly, for every policy $\pi$, it is known that $V^\pi$ and $ J^\pi$ satisfy the following \emph{Bellman equations} \citep[Theorem~8.2.6] {10.5555/528623}:    
\begin{align*}
  &P^\pi J^\pi=J^\pi , \qquad r^\pi+ P^\pi V^\pi = J^\pi+V^\pi.
\end{align*} 

    
    
    
    
    
\paragraph{Classification of MDPs.}
  MDPs are classified as follows by the structure of transition matrices. (For definitions of basic concepts of MDPs such as irreducible class, communicating class, accessibility, etc., please refer to \cite[Appendix A.2]{10.5555/528623} and \cite[Section 3]{bremaud2013markov}.) An MDP is \emph{ergodic} if the transition matrices induced by every policy $\pi$ are irreducible and aperiodic.     MDP is \emph{unichain} if the transition matrix corresponding to every deterministic policy consists of a single irreducible recurrent class plus a possibly empty set of transient states. MDP is \emph{weakly communicating} if there exists a closed set of states where each state in that set is accessible from every other state in that set under some deterministic policy, plus a possibly empty set of states which is transient under every policy. MDP is \emph{multichain} if the transition matrix corresponding to any deterministic policy contains one or more irreducible recurrent classes. Every ergodic MDP is unichain, every unichain MDP is weakly communicating, and every weakly communicating MDP is multichain. We use the expressions \emph{multichain} and \emph{general} interchangeably. 

 If an MDP is unichain,  for every policy $\pi$, there exists a unique stationary distribution $g^\pi \in \cM(\cS)$ satisfying $(g^\pi)^\top P^\pi=(g^\pi)^\top$  \cite[Theorem~8.3.2]{10.5555/528623}.  If an MDP is ergodic, for every policy $\pi$, $g^\pi$ is strictly positive  \cite[Proposition 1.19]{levin2017markov}. If an MDP is unichain, for every policy $\pi$, $J^\pi$ is a uniform constant vector, i.e.,
$J^\pi = c\mathbf{1}$ for some $c\in \mathbb{R}$ where $\mathbf{1}\in \mathbb{R}^{|\cS|}$ is the vector with entries all $1$.   
 In the next section, we discuss the continuity of the mapping $\pi\mapsto J^\pi$ in multichain MDP. Since $|\cS|$ and $|\cA|$ are finite, we can identify $\pi$ and $J^\pi$ as finite-dimensional vectors, namely, as $\pi\in \mathbb{R}^{|\cS|\times |\cA|}$ and $J^\pi\in \mathbb{R}^{|\cS|}$.
Therefore, continuity of $\pi\mapsto J^\pi$ can be interpreted as continuity of the mapping from $\mathbb{R}^{|\cS|\times |\cA|}$ to $\mathbb{R}^{|\cS|}$ under the usual metric. We define the class of policies
\begin{alignat*}{3}
\Pi = \text{set of all policies}=\cM(\cA)^\cS, \qquad \Pi_{+} = \{ \pi\in \Pi \,|\, \pi (a \,|\, s) >0 \text{ for all } (s,a) \in \cS\times \cA\},
\end{alignat*}
so that $\Pi_+$ is the (relative) interior of $\Pi$. Also define $\cM_+(\cX)$ as the space of probability distributions with full support over a set $\cX$.
\section{Recurrent-transient theory of policy gradients}\label{sec:recurrent_transeint}
In this section, we apply the recurrent-transient theory of Markov chains to the average-reward multichain MDP setting and introduce the notions of the \emph{recurrent visitation measure} and the \emph{transient visitation measure}. Building on these concepts, we establish a policy gradient theorem for average-reward multichain MDPs.
\subsection{Recurrent-transient classification of states}
\begin{definition}
Given a policy $\pi\in \Pi$, 
a state $s\in \cS$ is recurrent if its return time starting from $s$ is finite with probability $1$. Otherwise, $s$ is transient.
\end{definition}
Equivalently, if $n_s$ is the random variable representing the number of visits to state $s$ starting from $s$, then $s$ is recurrent if and only if $\expec_\pi[n_s]=\sum_{k=0}^{\infty} (P^\pi)^{k}(s, s)=\infty,
$ and otherwise it is transient \cite[Theorem 3.1.3]{bremaud2013markov}. 

Let $\pi \in \Pi$. For a given $P^\pi$, the states can be classified into recurrent and transient states, and the Markov chain can be canonically represented as follows \cite[Appendix A.2]{10.5555/528623}
$$
P^{\pi} =
\begin{bmatrix}
R^{\pi}_1 & 0 & 0 & \cdot & \cdot & 0 \\
0 & R^{\pi}_2  &0 & \cdot & \cdot & 0 \\
 &   & & \cdot & \cdot & 0 \\
0 & 0 & 0  & \cdot & R^{\pi}_m & 0 \\
S^{\pi}_1 & S^{\pi}_2 & \cdot  & \cdot & S^{\pi}_m &T^{\pi}
\end{bmatrix}, \qquad P^\pi_\star =
\begin{bmatrix}
R^{\pi}_{1,\star} & 0 & 0 & \cdot & \cdot & 0 \\
0 & R^{\pi}_{2,\star}  &0 & \cdot & \cdot & 0 \\
0 &   & & \cdot & \cdot & 0 \\
0 & 0 & 0  & \cdot &  R^{\pi}_{m,\star} & 0 \\
S^{\pi}_{1,\star} & S^{\pi}_{2,\star} & \cdot  & \cdot & S^{\pi}_{m,\star} & 0
\end{bmatrix},
$$
where $R_1^\pi,\dots,R_m^\pi,T^\pi,S_1^\pi,\dots,S_m^\pi$ represent transition probabilities among the recurrent states, among the transient states, and from transient to recurrent states, respectively. (Note that a multichain MDP can have $m\ge 1$ recurrent classes whereas a unichain MDP must have a single recurrent class). $P^\pi_\star$ is the Ces\`aro limit of $P^\pi$, and it is known that $R^{\pi}_{i,\star} = \mathbf{1} (g^\pi_i)^\top$ where $g^\pi_i$ is the unique stationary distribution of the stochastic matrix $R^{\pi}_{i}$ and $S^\pi_{i,\star} = (I-T^\pi)^{-1}S^\pi_i R^{\pi}_{i,\star}$ for $i=1,\dots,m$ \cite[Appendix A.4]{10.5555/528623}. 

Such canonical representations exist for any $\pi \in \Pi$, but the recurrent-transient classification of states and the structure of the corresponding canonical decomposition of $P^\pi$ may vary as a function of $\pi\in\Pi$. However, the recurrent-transient classification remains invariant for all $\pi \in \Pi_+$. 
\begin{fact}\protect{\cite[Proposition~1]{lee2025policy}}\label{fact:classification}
The recurrent-transient classification of the states does not depend on the choice of $\pi\in \Pi_+$. More specifically, the state set $\mathcal{S}$ decomposes into the transient states $\mathcal{T}$ and $m$ recurrent classes $\mathcal{R}_1,\dots,\mathcal{R}_m$ for some $m\in \mathbb{N}$, and this decomposition is the same across all $\pi\in \Pi_+$. 
\end{fact}

We further clarify. For any $\pi\in \Pi_+$, the recurrent-transient classification is determined by the transition kernel $P$, not by the particular choice of $\pi\in \Pi_+$. While a policy $\pi\in \Pi\setminus\Pi_+$ (a policy that assigns zero probability to some actions) may induce a different classification, the algorithms we consider (as well as deep RL algorithms employing a softmax output layer) search over $\Pi_+$ rather than the full set $\Pi$. 

By Fact~\ref{fact:classification}, we therefore fix, for any $\pi \in \Pi_+$, the collection of recurrent classes $\bigcup_{i=1}^m \mathcal{R}_i \subset \mathcal{S}$ and the transient set $\mathcal{T}\subset \mathcal{S}$, where $m$ denotes the number of recurrent classes. We also write
\[
\mathcal{R} := \bigcup_{i=1}^m \mathcal{R}_i
\]
for the set of all recurrent states. 

We also note that even though we borrow Fact~\ref{fact:classification} from \cite{lee2025policy}, the analyses of recurrent and transient states are substantially different. Specifically, \cite{lee2025policy} studies the undiscounted infinite-horizon total-reward setting where the main focus is on the dynamics of transient states, whereas in our multichain MDP setting, both recurrent and transient states must be analyzed, as Theorem~\ref{thm::policy_gd} shows.

\subsection{Continuity of $J^\pi$ on $\Pi_+$}
For multichain MDPs, the average-reward $J^\pi(s)$ can be a discontinuous function of the policy $\pi\in \Pi$ as the counterexample in \cite[Section 8]{schweitzer1968perturbation} demonstrates. However, we do have continuity on $\pi\in \Pi_+$.



\begin{lemma}\label{lem:continuity}
In a multichain MDP, the mappings $\pi \mapsto J^\pi$ and $\pi \mapsto J^\pi_\mu$
(where $J^\pi_\mu = \mathbb{E}_{s\sim\mu}[J^\pi(s)]$)
are continuous on $\Pi_+$ for any fixed $\mu \in \mathcal{M}(\mathcal{S})$.
\end{lemma}
In other words, we have continuity on the (relative) interior of $\Pi$ and the discontinuity described in \cite[Section 8]{schweitzer1968perturbation} can arise only on the (relative) boundary of $\Pi$. In our policy gradient methods, we restrict the search to policies $\pi \in \Pi_+$.



Next, we define our performance measure as
\[
J_{+,\mu}^\star=
\sup_{\pi \in \Pi_+} \mathbb{E}_{s\sim\mu} [J^\pi(s)],\qquad
 J_{\mu}^\star=
 \max_{\pi \in \Pi}\mathbb{E}_{s\sim\mu} [J^\pi(s)],
\]
where $\mu$ is an initial state distribution.
By definition, $J_{\mu}^\star \ge J_{+,\mu}^\star$. Since our policy gradient methods search over $\Pi_+$, they should be thought of as optimizing for $J_{+,\mu}^\star$. Thus, given $\mu \in \cM(\cS)$, we say a policy $\pi$ is an $\epsilon$-optimal policy with respect to positive policies if $J_{+,\mu}^{\star}-J_\mu^{\pi} \le \epsilon$.

Our performance measure,  the expected value function under an initial-state distribution, has been widely used as a conventional objective in the discounted-reward setting \cite{agarwal2021theory} and in the stochastic optimization literature \cite{nemirovski2009robust}. We believe that our measure  for multichain MDPs is a natural extension of the unichain case \cite{sutton1999}, but other performance measures, such as vector-valued or chain-dependent performance measures, could be more useful depending on the problem at hand.




\subsection{Visitation measures}
In prior works, the state visitation measure for average-reward unichain MDPs under a policy $\pi$ is defined as the stationary distribution, which is independent of the initial state $s_0$. In contrast, for multichain MDPs, the stationary distribution generally depends on the starting state, and it might not be unique. To address this, we introduce new visitation measures for the multichain setting.

First, define the \emph{recurrent visitation measure} with respect to policy $\pi$ and starting state $s_0$ as
\[d_{s_0}^\pi (s) = \lim_{H\rightarrow \infty} \frac{1}{H}\sum_{h=0}^{H-1}\mathrm{Prob}(s_h = s \mid s_0;\, P^\pi)=e_{s_0}^{\intercal}P_\star^\pi e_{s},
\]
where $e_s$ is the $s$-th unit vector.
Likewise, define the \emph{recurrent visitation measure} with respect to policy $\pi$ and starting state distribution $\mu$ as
\[
d^\pi_{\mu}= \expec_{s_0 \sim \mu} [d^\pi_{s_0}].
\]
Note that, for all transient states $s$, we have $d^\pi_{s_0}(s)=d^\pi_{\mu}(s)=0$.



Next, following \citep{lee2025policy}, define the \emph{transient matrix}:
\[
\bar{T}^\pi=\begin{bmatrix}
    0&0\\0&T^\pi
\end{bmatrix}, 
\qquad
\text{i.e.,}
\qquad
\bar{T}^\pi(s_1,s_2)=
\left\{\begin{array}{ll}
P^\pi(s_1,s_2)&\text{if $s_1,s_2$ are both transient}\\
0&\text{otherwise.}
\end{array}\right.
\]
\begin{fact}\cite[Lemma 8.3.20]{berman1994nonnegative}\label{fact:trans_spec}
Spectral radius of $\bar{T}^\pi$ is strictly less than $1$.
\end{fact}
By Fact~\ref{fact:trans_spec} and the classical Neumann series argument, we have $(I-\bar{T}^\pi)^{-1} = \sum_{i=0}^\infty (\bar{T}^\pi)^i$. 
Now, we define the \emph{transient visitation measure} with respect to policy $\pi$ and starting state $s_0$ as
\[\delta^\pi_{s_0}(s)
= \sum_{h=0}^{\infty}\mathrm{Prob}(s_h = s \mid s_0;\, \bar{T}^\pi)
= e_{s_0}^{\intercal}(I-\bar{T}^\pi)^{-1}e_{s},
\]
and the \emph{transient visitation measure} with respect to policy $\pi$ and starting state distribution $\mu$ as
\[
\delta^\pi_{\mu}= \expec_{s_0 \sim \mu} [\delta^\pi_{s_0}].
\]
We point out that this transient visitation measure is not a probability measure.
Also, importantly, $\max_{s, s_0 \in \cS} \delta^\pi_{s_0}(s) < \infty$ by Fact \ref{fact:trans_spec}.



With these recurrent and transient visitation measures, we can obtain the following performance difference lemma in the average-reward multichain MDP setup, which will be crucially used in the analysis of policy gradient algorithms.


\begin{lemma}[Performance difference lemma]\label{lem:pdl} Consider a multichain MDP.   For $\pi, \pi' \in \Pi_+$ and $\mu\in\mathcal{M}_+(\mathcal{S})$,  
\begin{align*}
J_\mu^{\pi} - J_\mu^{\pi'}
&\!= \! \sum_{\substack{s \in \cR\\a \in \cA}}
d^{\pi}_{\mu}(s) \big(\pi(a \mid s)-\pi'(a \mid s)\big)Q^{\pi'}(s,a) +
\sum_{\substack{s \in \cT\\a \in \cA}}\delta^{\pi}_{\mu}(s)\big(\pi(a \mid s)-\pi'(a \mid s)\big) PJ^{\pi'}(s,a).
\end{align*}
\end{lemma}
We note that \cite{cao2004unified} first studied a performance difference lemma for multichain average-reward MDPs. However, our Lemma~\ref{lem:pdl} directly expresses the difference in average reward in terms of policy differences, $Q^\pi$, and $PJ^\pi$ by leveraging the invariance of recurrent-state classifications, making it more suitable for policy-gradient analysis.

We also point out that Lemma~\ref{lem:pdl} can be viewed as a generalization of the following performance difference lemma for the unichain setup. 
\begin{corollary}{\cite[Lemma~4.1]{even2009online}}\label{cor:pdl}
Consider a unichain MDP.    For $\pi, \pi' \in \Pi_+$ and $\mu\in \mathcal{M}_+(\mathcal{S})$, 
\begin{align*}
J_\mu^{\pi} - J_\mu^{\pi'}
&=    \sum_{s \in \cR}
\sum_{a \in \cA}
d^{\pi}_{\mu}(s) \big(\pi(a \mid s)-\pi'(a \mid s)\big)Q^{\pi'}(s,a). 
\end{align*}
\end{corollary}
\begin{proof}
For any $\pi' \in \Pi_+$, the unichain assumption implies $PJ^{\pi'}=c^{\pi'}\mathbf{1}$ for some $c^{\pi'} \in \mathbb{R}$ \cite[Theorem~8.3.2]{10.5555/528623}. Therefore, the second term in  Lemma~\ref{lem:pdl} vanishes.
\end{proof}
\subsection{Policy gradient theorem}

We are now ready to present the policy gradient theorem in the average-reward setup. Consider the optimization problem
\[\max_{\theta \in \Theta }  J^{\pi_{\theta}}_{\mu},
\]
where $\{\pi_{\theta} \mid \theta \in \Theta\subset \real^d\}$ is a set of differentiable parametric policies with respect to $\theta$. Based on previous machinery, we establish the following policy gradient theorem.   
\begin{theorem}[Recurrent and transient policy gradient]\label{thm::policy_gd}
Consider a multichain MDP. For $\pi_{\theta} \in \Pi_{+}$  and $\mu\in \mathcal{M}_+(\mathcal{S})$,
\[\nabla_{\theta}  J_\mu^{\pi_\theta}
=    \sum_{s \in \cR}
\sum_{a \in \cA}
d^{\pi_\theta}_{\mu}(s) \nabla_{\theta} \pi_{\theta} (a \mid s) Q^{\pi_\theta}(s,a) 
 +\sum_{s \in \cT}
\sum_{a \in \cA}  \delta^{\pi_\theta}_{\mu}(s)\nabla_{\theta} \pi_{\theta} (a \mid s) PJ^{\pi_\theta}(s,a).
\]
\end{theorem}
    In Theorem~\ref{thm::policy_gd}, the first term involves a sum over recurrent states, while the second term sums over transient states. Viewed in this way, Theorem~\ref{thm::policy_gd} naturally generalizes the classical policy gradient theorem for unichain MDPs.
    
\begin{corollary}[Recurrent policy gradient, \cite{sutton1999}]\label{cor::policy_gd}
 Consider a unichain MDP. 
 For $\pi_{\theta} \in \Pi_{+}$  and $\mu\in \mathcal{M}_+(\mathcal{S})$,
\[\nabla_{\theta}  J_\mu^{\pi_\theta}
=    \sum_{s \in \cR}
\sum_{a \in \cA}
d^{\pi_\theta}_{\mu}(s) \nabla_{\theta} \pi_{\theta} (a \mid s) Q^{\pi_\theta}(s,a). \]
\end{corollary}
\begin{proof}
 For any $\pi \in \Pi_+$, the unichain assumption implies $PJ^{\pi}=c^{\pi}\mathbf{1}$ for some $c^{\pi} \in \mathbb{R}$ \cite[Theorem~8.3.2]{10.5555/528623}. By the product rule, $
\sum_{a \in \cA}\nabla_{\theta} \pi_{\theta} (a \mid s) PJ^{\pi_\theta}(s,a) =\nabla_{\theta} (\sum_{a \in \cA} \pi_{\theta} (a \mid s) PJ^{\pi_\theta}(s,a))- \sum_{a \in \cA} \pi_{\theta} (a \mid s) \nabla_{\theta} PJ^{\pi_\theta}(s,a) =\nabla_{\theta} (c^{\pi_{\theta}})-\nabla_{\theta} (c^{\pi_{\theta}})=0$. Therefore, the second term in Theorem~\ref{thm::policy_gd} vanishes.
\end{proof}





\section{Convergence of policy mirror ascent}\label{sec:3}

In this section, we study the convergence of policy mirror ascent on a compact subset of $\Pi_+$.

\subsection{$\alpha$-clipped policy mirror ascent}
We consider the direct parameterization
\[
\pi_{\theta} (a \,|\, s) = \theta_{s,a},
\]
where $\theta \in \real^{|\cS|\times |\cA|}$ satisfies $\sum_{a \in \cA} \theta_{s,a}=1$ and $\theta_{s,a}\ge 0$ for all $s\in \cS$ and $a\in \cA$. With this direct parameterization, we do not distinguish between the policy $\pi_\theta$ and the parameter $\theta$, and we use $\pi_k$ to denote the iterates of the algorithm.

Let $G^\pi(s,\cdot)= Q^{\pi}(s,\cdot) \mathbf{1}_{ \cR}(s) + PJ^\pi(s,\cdot)\mathbf{1}_{\cT} (s)$ for all $s\in \cS$, and let $D(\cdot, \cdot): \cM(\cA)\times \rint\,\cM(\cA)\rightarrow \real$ be a Bregman divergence. Following the derivations in \cite{shani2020adaptive} and \cite{xiao2022convergence}, together with our policy gradient theorem, we can obtain the following update formulation (please refer to Appendix~\ref{appen:miss_3} for a detailed derivation). 
\[
\pi_{k+1}(\cdot \, |\, s)
=
\argmax_{p\in \mathcal{C}}
\Big\{
\ \sum_{a\in \cA}\eta_kG^{\pi_k}(s,a)p(a)
- D\big(p(\cdot),\,\pi_{k}(\cdot \,|\, s)\big)
\Big\}, \qquad \forall s \in \cS.
\]
The choice $\mathcal{C}=\cM (\cA)$ may seem natural, but in our setup, we further restrict the policy set to guarantee finiteness of certain coefficients appearing in the convergence analysis. Define
\[
\Pi_{\alpha} = \{ \pi \,|\, \pi (a \,|\, s) \ge \alpha\text{ for all }s\in \cS,\,a\in \cA\}, \qquad \cM_\alpha(\cA) =\{p \,|\, p(a) \ge \alpha \text{ for all }\,a\in \cA\}
\]
with  $\alpha \in (0,1/|\cA|)$. Indeed, on $\Pi_\alpha$, we can ensure finiteness of the following coefficients.  
For $\pi\in\Pi_+$, define
\[
\rho_\mu^\pi(s)=d_\mu^\pi(s)\mathbf{1}_{\cR}(s)+\delta_\mu^\pi(s)\mathbf{1}_{\cT}(s),
\]
where $\mathbf{1}_{\cR}$ and $\mathbf{1}_{\cT}$ are indicator functions.
\begin{lemma}\label{lem:bound}
If $\mu\in \mathcal{M}_+(\mathcal{S})$ and $\alpha \in (0,1/|\cA|)$, then
    \[\max_{\pi \in {\Pi_\alpha}}\|\rho^\pi_{\mu}\|_1:=B_\alpha< \infty, \quad \max_{s \in \cS, \pi, \pi' \in \Pi_\alpha} \infn{\frac{\rho^{\pi}_\mu(s)}{\rho_\mu^{\pi'}(s)}}:=C_{\alpha}<\infty, \quad Q_{\alpha}:= \max_{\pi \in \Pi_\alpha}\infn{Q^{\pi}}<\infty\]  
\end{lemma}

Now, we are ready to present the algorithm $\alpha$-clipped policy mirror ascent.
 
 \begin{algorithm}[H]
\caption{$\alpha$-clipped policy mirror ascent}
\label{alg:spmd}
\begin{algorithmic}
\STATE \textbf{Input:} 
$ \alpha \in (0,1/|\mathcal{A}|), K, \pi_0\in \Pi_\alpha, \{\eta_k\}^{K-1}_{k=0} \subset (0,\infty), D(\cdot, \cdot)$
\FOR{$k=0,1,\ldots,K-1$}
\FOR{$s \in \cS$}
        \STATE $\pi_{k+1}(\cdot \,|\, s)=
 \argmax_{p\in \cM_\alpha (\cA)}
 \left\{
 \eta_k \sum_{a \in \cA}G^{\pi_k}(s,a)p(a)
 - D\big(p(\cdot),\pi_{k}(\cdot \,|\, s)\big)\right\}$
\ENDFOR
\ENDFOR
\STATE \textbf{Output:} $\pi_K$ 
\end{algorithmic}
\end{algorithm}

 If $h$ is a Legendre function (please refer to Appendix~\ref{app:pre} for definition), such as the half squared Euclidean norm $h(p)=\frac12\|p\|_2^2$ or negative entropy, $\pi_{k+1}$ always exists \citep[Lemma~6]{xiao2022convergence}. 
 Specifically, if the Bregman divergence is half the squared Euclidean distance, $\alpha$-clipped policy mirror ascent is reduced to
  \[    \pi_{k+1}(\cdot \,|\, s) = \argmin_{p\in \cM_\alpha}\left\|p(\cdot)-(\pi_{k}(\cdot \,|\, s)+\eta_k G^{\pi_k}(\cdot \,|\, s))  \right\|^2_2, \qquad \forall\, s \in \cS\]
and if Bregman divergence is KL-divergence, it is reduced to
    \[     \pi_{k+1}(\cdot \,|\, s) = \argmin_{p\in \cM_\alpha} \text{KL}\big(p(\cdot) \,|\,\pi_{k}(\cdot \,|\, s) \exp(\eta_k G^{\pi_k}( s, \cdot))/Z_s \big), \qquad \forall \,s \in \cS,
    \]
where  $Z_s=\sum_{a\in \cA}\pi_{k}(a \,|\, s) \exp(\eta_k (G^{\pi_k}( s, a))$. If $\alpha=0$, these methods are known as the projected $Q$-descent and multiplicative weights updates, respectively \citep{xiao2022convergence, freund1997decision}. We note that both optimization subproblems can be solved through Euclidean and KL projection algorithms with $\widetilde{\mathcal{O}}(|\cS||\cA|) $ time complexity \citep{wang2013projection, herbster2001tracking}. We present both projection algorithms in Appendix~\ref{appen:proj}.

In the following subsections, we establish convergence of $\alpha$-clipped policy mirror ascent. For a given $\mu$, define \[\pi_\alpha\in \argmax_{\pi \in \Pi_{\alpha} } J_\mu^\pi.\] 
By Lemma~\ref{lem:continuity}, $\pi_\alpha$ exists, and we will show convergence of $J^{\pi_k}$ with respect to $J^{\pi_{\alpha}}$ with $\pi_k\in \Pi_\alpha$.

\subsection{Sublinear convergence with constant step size}

We first establish the sublinear convergence of the policy gradient algorithm with a constant step size. 
As a first step in our analysis, we state the following lemma, which ensures that the policies generated by the $\alpha$-clipped policy mirror ascent improve monotonically.

\begin{lemma}\label{lem:ngd_descent}
Consider a multichain MDP. For $\pi_0 \in \Pi_\alpha$ and $\mu\in \mathcal{M}_+(\cS)$, the $\alpha$-clipped policy mirror ascent  generates a sequence of policies $\{\pi_{k}\}^\infty_{k=1}$ satisfying
\[
J_\mu^{\pi_k} \le J_\mu^{\pi_{k+1}}.
\]
\end{lemma}

Through Lemma~\ref{lem:ngd_descent}, we obtain the following convergence result of the $\alpha$-clipped policy mirror ascent. The proof presented in Appendix~\ref{appen:miss_3} closely follows \cite{xiao2022convergence} which considers the discounted reward setup and uses the discounted state-visitation distribution.

\begin{theorem}\label{thm:ngd}
Consider a multichain MDP. For $\pi_0 \in \Pi_\alpha$ and $\mu\in\cM_+(\cS)$, the $\alpha$-clipped policy mirror ascent with constant step size $\eta>0$ generates a sequence of policies $\{\pi_{k}\}^\infty_{k=1}$ satisfying
\[
J_\mu^{\pi_\alpha}-J_\mu^{\pi_k} 
\le \frac{1}{k+1}\bigg(\frac{D_{\rho_\mu^{\pi_\alpha}}(\pi_\alpha,\pi_0)}{\eta}
+C_\alpha (J_\mu^{\pi_\alpha}-J_{\mu}^{\pi_{0}}) \bigg).
\]
\end{theorem}
Theorem~\ref{thm:ngd} shows that $J_\mu^{\pi_k}\rightarrow J_\mu^{\pi_\alpha}$ with a sublinear rate, and since $J^{\pi_{\alpha}}_\mu \rightarrow J^\star_{+,\mu}$  as $\alpha\rightarrow 0$, if we choose a sufficiently small $\alpha$ such that $J_{+,\mu}^\star-J_{\mu}^{\pi_\alpha}< \frac{\epsilon}{2}$, we can obtain policy $\pi$ satisfying $J_{+,\mu}^\star-J_\mu^\pi <\epsilon$ through Theorem~\ref{thm:ngd}. 

\subsection{Linear convergence with adaptive step size}

Next, we present the \emph{linear} convergence rate of the $\alpha$-clipped policy mirror ascent with adaptive step sizes. 
If $C_\alpha=1$, replace it by any fixed upper bound greater than $1$, such as $2$, in all adaptive-step-size statements and proofs below.
\begin{theorem}\label{thm:ngd_linear}
Consider a multichain MDP. For $\pi_0 \in \Pi_\alpha$ and $\mu\in \cM_+(\cS)$, the $\alpha$-clipped policy mirror ascent with step sizes $(C_\alpha-1)\eta_{k+1}\ge C_\alpha\eta_k>0$ generates a sequence of policies $\{\pi_{k}\}^\infty_{k=1}$ satisfying   
\[
J_\mu^{\pi_\alpha}-J_\mu^{\pi_{k}} 
\;\le\;
\left(1-\frac{1}{C_\alpha}\right)^{\!k}
\bigg( \frac{D_{\rho_\mu^{\pi_\alpha}}(\pi_\alpha, \pi_{0})}{\eta_0(C_\alpha-1)}
+
J_\mu^{\pi_\alpha}-J_\mu^{\pi_0}\bigg).
\]
\end{theorem}



Although the adaptive step size yields a linear convergence rate, it requires knowledge of $C_{\alpha}$ to set the step sizes. In contrast, a constant step size always guarantees a sublinear rate. We point out that to the best of our knowledge, Theorems~\ref{thm:ngd} and \ref{thm:ngd_linear} are the first convergence results of policy gradient methods for average-reward multichain MDPs.

\subsection{Experiments}

\begin{figure}
    \centering
    \includegraphics[width=1\linewidth]{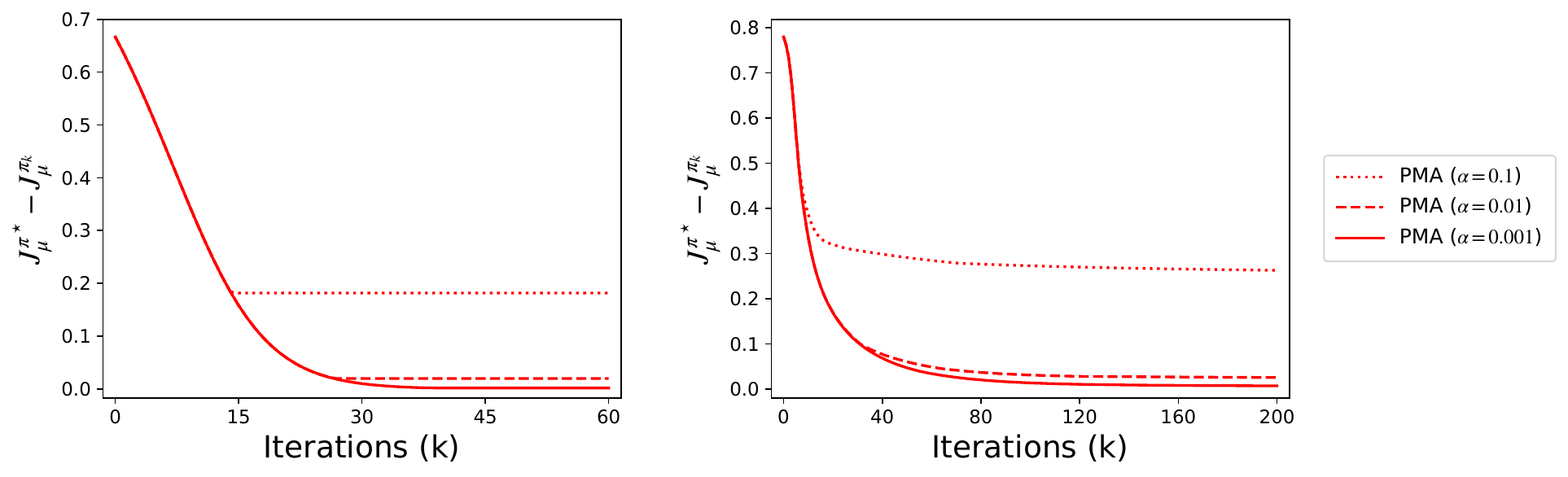}
    \vspace{-0.15in}
    \caption{Comparison of $\alpha$-clipped policy mirror ascent (PMA) in (left) Toll Collector and (right) Frozen Islands. The limit of the policy mirror ascent algorithm
gets closer to the optimum as $\alpha$ gets smaller}
\vspace{-0.1in}
    \label{fig_1}
\end{figure}
For the experiments, we consider two toy multichain MDPs: The Toll Collector problem \cite[Section 7.4]{atia2021steady} with 10 states and three actions at the entry state, and two actions at each remaining state, and an asymmetric Frozen Islands environment \cite[Section 4]{atia2020steady} with \(8\times8\) states and four actions. Under positive policies, Toll Collector has one transient state and three recurrent classes, and Frozen Islands has 32 transient states and two recurrent classes. Both instances satisfy $\lim_{\alpha\rightarrow 0}J^{\pi_\alpha}_\mu= J^{\star}_{+,\mu}=J^{\pi_\star}_\mu$, and we use a uniform initial state distribution.

We run \(\alpha\)-clipped policy mirror ascent with KL divergence for \(\alpha\in\{0.1,0.01,0.001\}\), using a constant step size \(\eta=0.1\) in both environments. The initial policy is a uniform distribution over the available actions at each state. All updates use exact tabular policy evaluation with \(G^\pi\). We perform 60 iterations for Toll Collector and 200 iterations for Frozen Islands, and plot the policy optimality gap \(J_\mu^{\pi^\star}-J_\mu^{\pi_k}\) against the iteration number \(k\).  The results are shown in Figure~\ref{fig_1}. The \(\alpha\)-clipped policy mirror ascent converges to $J_\mu^{\pi_\alpha}$ for each $\alpha$, and a smaller $\alpha$ makes policy mirror ascent converge closer to $J_\mu^{\pi^\star}$, since $J_\mu^{\pi_\alpha}$ increases monotonically to $J_\mu^{\pi^\star}$ as $\alpha\to 0$.

Lastly, we note that for Toll Collector problem
\[
C_\alpha = \frac{2(1-\alpha^2)}{\alpha(3+\alpha)},
\qquad
B_\alpha = \frac{11}{10},
\qquad
Q_\alpha =
\frac{8-11\alpha+4\alpha^2+4\alpha^3-\alpha^4}
{(2-\alpha)^2(1+\alpha)^2},
\qquad
J_\mu^{\pi_\alpha} = \frac{1-\alpha}{1+\alpha}.
\]
So  $C_\alpha \sim \frac{2}{3\alpha},
B_\alpha < 2,
Q_\alpha < 2,
J_\mu^\star - J_\mu^{\pi_\alpha}
= \frac{2\alpha}{1+\alpha}$. Note that this is  empirical evaluation for this instance and we acknowledge that our current theoretical analysis does not provide a sensitivity characterization of these $\alpha$-dependent constants.

\section{Sample complexity of policy mirror ascent}\label{sec:4}

In this section, we extend the analysis of the $\alpha$-clipped policy mirror ascent to the sampling setting in which the transition probabilities are unknown. Specifically, we assume access to a generative model \citep{kearns1998finite}, which provides independent samples of the next state for any given state and action.

\begin{algorithm}[H]
\caption{Critic  }
\label{alg:critic}
\label{alg:multi_traj_gm}
\begin{algorithmic}
\STATE \textbf{Input:} $\pi \in \Pi_+, N,H,N', H'\in\mathbb{Z}_{+}$
\FOR{$i=1,\ldots,N$}
\FOR{$(s_0,a_0) \in \cS \times \cA$}
  \STATE Generate $\{(s_0,a_0),(s_1,a_1),\ldots,(s_H,a_H)\,|\, s_h \sim P(\cdot\,|\, s_{h-1}, a_{h-1}), a_{h} \sim \pi(\cdot \,|\, s_{h}) \}$ 
  \STATE $\widehat{PJ}^{i}(s_0,a_0)=\frac{1}{H+1}\sum^H_{h=0}r_i(s_h,a_h)$.
   \ENDFOR
    \ENDFOR
\STATE $\widehat{PJ}^{\pi}=\frac{1}{N}\sum_{i=1}^{N}\widehat{PJ}^{i}$.
\FOR{$i=1,\ldots,N'$}
\FOR{$(s_0,a_0) \in \cS\times\cA$}
  \STATE Generate $\{(s_0,a_0),(s_1,a_1),\ldots,(s_{H'},a_{H'})\,|\, s_{h'} \sim P(\cdot\,|\, s_{h'-1}, a_{h'-1}), a_{h'} \sim \pi(\cdot \,|\, s_{h'}) \}$ 
    
    \STATE $\widehat{Q}^{i}(s_0,a_0)=\frac{1}{H'+1}\sum_{h=0}^{H'} \sum^h_{j=0}\big(r_i(s_j,a_j)-\widehat{PJ}^{\pi}(s_j,a_j)\big)$.
  \ENDFOR
  \ENDFOR
  \STATE $\widehat{Q}^{\pi}=\frac{1}{N'}\sum_{i=1}^{N'}\widehat{Q}^{i}$
\FOR{$ s \in \cS$}
  \STATE $\widehat{G}^\pi(s,\cdot)= \widehat{Q}^{\pi}(s,\cdot) \mathbf{1}_{ \cR}(s) + \widehat{PJ}^\pi(s,\cdot)\mathbf{1}_{\cT} (s)$ 
  \ENDFOR
\STATE \textbf{Output:} $\widehat{G}^\pi$
\end{algorithmic}
\end{algorithm}

\subsection{Approximating $G^\pi$ with a generative model}
With a generative model, for a given policy $\pi$ and any state-action pair $(s,a)\in\mathcal{S}\times\mathcal{A}$, we can
generate independent trajectories of horizon $H$:
\[
\left\{
\bigl(s_0=s,a_0=a\bigr),\bigl(s_1,a_1\bigr),\ldots,\bigl(s_{H-1},a_{H-1}\bigr)
\right\}.
\]
 With such samples, we approximate $G^\pi(s,\cdot)= Q^{\pi}(s,\cdot) \mathbf{1}_{ \cR}(s) + PJ^\pi(s,\cdot)\mathbf{1}_{\cT} (s)$, $ \forall s\in \cS$, through the \emph{critic} method described as Algorithm~\ref{alg:critic}.


To present the sample-complexity results for the critic method under a generative model, we define \emph{expected target time} as follows. For $\pi \in \Pi_+$ and each recurrent class $\mathcal{R}_i$,  
\begin{align*}
t^{\pi, i}_{\mathrm{tar}}
:= \sum_{s'\in\mathcal{R}_i} \rho_i^{\pi}(s')\,\mathbb{E}_{\pi}\!\left[t^i_{s'} \,|\, s_0=s \right]\text{ for }i=1,\dots,m,   \,\quad \,t^\pi_{\mathrm{tar}}=\max_{1\le i \le m}t^{\pi, i}_{\mathrm{tar}}, \,\quad\, t_\alpha=\max_{\pi \in \Pi_\alpha}t^{\pi}_{\mathrm{tar}}
\label{eq:Ctar}
\end{align*}
where $t_{s'} := \inf\{t\ge 0 : s_t = s'\}$ for $s, s'\in\mathcal{R}_i$, and $\rho_i^\pi$ is stationary distribution of $R^\pi_i$. It is known that expected target time is always finite and independent of the starting state when the state space is finite \citep{levin2017markov}, and  $t_\alpha$ is also finite as shown in the appendix. 
\begin{theorem}\label{thm:PE}
  Consider a multichain MDP.  Let $\epsilon, \delta\in(0,1)$.  For $\pi_0, \pi \in \Pi_+$ and $\mu\in \cM_+(\cS)$, with $1-\delta$ probability, the output of critic method satisfies $\|\widehat{G}^\pi -G^\pi\|_\infty \le \epsilon $ with sample complexity 
    \[\mathcal{O}\bigg( \frac{(1+R)^3(1+t^\pi_{\mathrm{tar}})^3(1+\infn{Q^\pi})^4|\cS||\cA|}{ \epsilon^6}\log\Big(\frac{2|\cS||\cA|}{\delta}\Big)\bigg).\]
\end{theorem}

\subsection{Observing classification of states through sampling }
To compute $\widehat{G}^\pi$ with $\widehat{Q}^\pi, \widehat{PJ}^\pi$ from the critic method, the classification of the states is required.

Define \emph{transient half-life} and \emph{cover time} for given $\pi$ and each recurrent class $\mathcal{R}^\pi_i$ as
\[t_{\frac{1}{2},\pi} = \min \left\{t \ge 1: \infn{(\bar{T}^\pi)^t}\le \frac{1}{2} \right\}  , \quad t^{\pi,i}_{\mathrm{cov}}=\max_{s\in \cR_i} \expec_\pi [t^i_{\mathrm{cov}} \,|\, s_0=s] \quad \text{and} \quad t^\pi_{\mathrm{cov}}=\max_{1\le i \le m} t^{\pi, i}_{\mathrm{cov}}\]
where $\bar{T}^\pi$ is transient matrix and $t^{i}_{\mathrm{cov}}:=\inf\{t\ge 0: \{s_j\}^{t-1}_{j=0}=\mathcal{R}_i \}$. It is known that these quantities are finite when the state space is finite \citep{lee2025policy, levin2017markov}. 
\begin{lemma}\label{lem:class}
Consider a multichain MDP.  Let $ \delta\in(0,1)$. Given $\pi \in \Pi_+$, set $M_1=t_{\frac{1}{2},\pi}\lceil\log_2(\frac{2}{\delta})\rceil$ and $M_2=\lceil3t^\pi_{\mathrm{cov}}\rceil\lceil\log(\frac{2}{\delta})\rceil$. With probability $1-\delta$, for a generated single trajectory with length $M_1+M_2$,  $(s_0, a_0, s_1, a_1, \dots, s_{M_1+M_2-1}, a_{M_1+M_2-1})$, $\{s_j\}^{M_1+M_2-1}_{j=M_1} = \mathcal{R}^\pi_i$ for some $i$.
\end{lemma}
This lemma shows that by sampling a trajectory of length $M_1 + M_2$ starting from an arbitrary state $s_0$, we can recover the complete set of states in a recurrent class. 
Moreover, by checking whether $s_0$ belongs to this set, we can determine the classification of $s_0$. Repeating this procedure for the remaining unclassified states under a generative model allows us to identify the classification of all states with sample complexity
$
\mathcal{O} \left((|\mathcal{T}| + m)\,(t_{\frac{1}{2},\pi} + t^\pi_{\mathrm{cov}})\,
\log\!\left(\frac{|\mathcal{T}| + m)}{\delta}\right)\right),
$
where $m$ denotes the number of recurrent classes.

\subsection{Sample complexity of  stochastic $\alpha$-clipped policy mirror ascent}

Based on results in previous sections, we now present stochastic $\alpha$-clipped policy mirror ascent:
\begin{algorithm}[H]
\caption{Stochastic $\alpha$-clipped policy mirror ascent} 
\label{alg:spmd}
\begin{algorithmic}
\STATE \textbf{Input:} 
$ \alpha \in (0,1/|\mathcal{A}|), K, \pi_0\in \Pi_\alpha, \{\eta_k\}^{K-1}_{k=0} \subset (0,\infty),D(\cdot, \cdot), \{N_k, H_k, N'_k,  H'_k\}^{K-1}_{k=0}$
\FOR{$k=0,1,\ldots,K-1$}
    \STATE 
     $\widehat{G}^{\pi_k}
        =\text{Critic}(\pi_k, N_k, H_k, N'_k, H'_k)$
        \FOR{$s \in \cS$}
      \STATE $\pi_{k+1}(s,\cdot)=
\argmax_{p\in \cM_\alpha (\cA)} \left\{\eta_k \sum_{a \in \cA}\widehat{G}^{\pi_{k}}(s,a)p(a)-\,D\big(p(\cdot),\pi_{k}(\cdot \,|\, s)\big)\right\}$
\ENDFOR
\ENDFOR
\STATE \textbf{Output:} $\pi_K$ 
\end{algorithmic}
\end{algorithm}


\begin{theorem}\label{thm:spmg}
Consider a multichain MDP. Let $\epsilon, \delta\in(0,1)$.
For any $\pi, \pi_0 \in \Pi_\alpha$ and $\mu\in \cM_+(\cS)$, with probability $1-\delta$, the output of stochastic $\alpha$-clipped policy mirror ascent with constant step size $\eta>0$ and $K=\max\left\{1,\left\lceil \frac{2A_\alpha}{\epsilon}\right\rceil\right\}$ satisfies $J_\mu^{\pi_\alpha}-J_\mu^{\pi_K}\le \epsilon$ with 
\[\widetilde{\mathcal{O}}\Bigg(\frac{(1+R)^3(1+t_\alpha)^3(1+Q_\alpha)^4 B_\alpha^6 (1+A_\alpha)(1+C_\alpha+A_\alpha)^6|\cS||\cA|}{\epsilon^{13}}  \Bigg)\]
 sample complexity  where
$A_\alpha=D_{\rho_\mu^{\pi_\alpha}}(\pi_\alpha,\pi_0)/\eta
+C_\alpha(J_\mu^{\pi_\alpha}-J_\mu^{\pi_0})$, and with adaptive step size $\eta_{k+1}(C_\alpha  - 1) \ge \eta_k C_\alpha >0$ and $K =\max\left\{1,\left\lceil\frac{\log\max\left\{1,\frac{2}{A'_\alpha\epsilon}\right\}}{\log(C_\alpha/(C_\alpha-1))}\right\rceil\right\}$ satisfies $J_\mu^{\pi_\alpha}-J_\mu^{\pi_K}\le \epsilon$ with 
\[\widetilde{\mathcal{O}}\bigg( \frac{(1+R)^3(1+t_\alpha)^3(1+Q_\alpha)^4 C_\alpha^7 B^6_\alpha|\cS||\cA|}{ \epsilon^6}\bigg)\]
sample complexity where $A'_\alpha =J_\mu^{\pi_\alpha}-J_\mu^{\pi_0}+\frac{D_{\rho_\mu^{\pi_\alpha}}(\pi_\alpha,\pi_0)}{\eta_0(C_\alpha-1)}$ and $\widetilde{\mathcal{O}}$ ignores all logarithmic factors including $\log(1/\delta)$.
\end{theorem}
We point out that to the best of our knowledge, Theorem~\ref{thm:spmg} is the first sample-complexity result of model-free methods for average-reward multichain MDPs. Lastly, we note that in Appendix~\ref{weak_mdp}, we provide refined analysis for weakly communicating MDPs .

\section{Conclusion}
In this work, we present the first analysis of policy gradient methods for average-reward multichain MDPs by establishing a policy gradient theorem in terms of recurrent and transient visitation measures. Our framework provides theoretical tools for extending analyses beyond the ergodic and unichain settings considered in most prior work. Just as policy gradient theorems in discounted MDPs \cite{sutton1999} have provided a foundation for developing and analyzing policy-based algorithms, we hope that our theoretical machinery will serve as the basis for such advances in the general average-reward setting.





\paragraph{Acknowledgments}
We are grateful to the anonymous reviewers and AC for their careful reading and constructive comments, which helped improve the presentation of the paper.

\bibliographystyle{abbrv}
\bibliography{neurips_2026}


\appendix

\newpage

\section{Preliminaries}\label{app:pre}
We say a function $h$ is essentially smooth if $h$ is differentiable and
$\|\nabla h(x_k)\|\to\infty$ for every sequence $\{x_k\}$ converging to a
boundary point of $\dom h$.
We say a function $h$ is of Legendre type if it is essentially smooth and strictly
convex in the (relative) interior of $\dom h$, 

\begin{fact}[Three-point descent lemma]\cite[Lemma 6]{xiao2022convergence}\label{fact:bregman}
Suppose that $\,\mathcal{C}\subset\mathbb{R}^n$ is a closed convex set,
$\phi:\mathcal{C}\to\mathbb{R}$ is a proper, closed, and convex function,
$D(\cdot,\cdot)$ is the Bregman divergence generated by a function $h$ of
Legendre type and $\rint \dom h \cap \mathcal{C}\neq\varnothing$.
For any $x\in \rint \dom h$, assume that the following minimizer exists and let
\[
x^{+}=\argmin_{u\in\mathcal{C}} \big\{\, \phi(u)+D(u,x) \,\big\}.
\]
Then, $x^{+}\in \rint \dom h \cap \mathcal{C}$ and for any $u\in\mathcal{C}$,
\[
\phi(x^{+}) + D(x^{+},x) \;\le\; \phi(u) + D(u,x) - D(u,x^{+}).
\]
\end{fact}

In our setup, $\mathcal{C}=\cM_
\alpha(\cA)$,  $\phi(p)=-\eta_k \langle G^{\pi_k}(\cdot,s), p(\cdot)\,\rangle$, and $h$ is the negative-entropy function, which is also
of Legendre type, satisfying $\rint\,\operatorname{dom}h \cap \mathcal{C}
= \cM_
\alpha(\cA) $.
Therefore, if we start with an initial point in $ \cM_
\alpha(\cA)$, then every
iterate will stay in $ \cM_
\alpha(\cA)$.

\begin{fact}[Hoeffding inequality]
      Let $X_1, \dots, X_n$ be independent random variables such that $ a_i \le X_i  \le b_i$ for all $i$. Then
\[ \mathbb{P}\left( \left|\frac{1}{n}\sum^n_{i=1} X_i- \expec \left[\frac{1}{n}\sum^n_{i=1} X_i\right]\right| \ge \epsilon\right) \le 2 \text{exp} \left(-\frac{2n^2 \epsilon^2}{\sum^n_{i=1}(b_i-a_i)^2}\right).\]
\end{fact}

\begin{fact}[Target time lemma]\cite[Corollary~3]{roberts1997shift}\label{fact:hittting} For $\pi \in \Pi_+$ and $\forall s \in \cR_i$, 
\[
\left\|
\frac{1}{k}\sum_{j=1}^{k}\bigl(R_i^{\pi}\bigr)^{j}(s,\cdot)\;-\;(g_i^{\pi})^\top
\right\|_{1}
\;\le\;
\frac{2 t^{\pi,i}_{\mathrm{tar}}}{k},
\]
where $g^\pi_i$ is stationary distribution of $R_i^{\pi}$.
\end{fact}

\section{Omitted proofs of Section 2}

\subsection{Proof of Lemma~\ref{lem:pdl}}
We first prove the following lemma. 
\begin{lemma}\label{lem:con}
    For any $\pi \in \Pi_+$, $a,a'\in \cA$ and $s,s' \in  \cR_i$, $PJ^\pi(s,a)=PJ^\pi(s',a')$ and $ J^\pi(s)=J^\pi(s')$.
\end{lemma}
\begin{proof}
 For $s \in \cR_i$ and $ s' \not\in \cR_i$, $P(s' \,|\, s,a )=0 $.  For any $s,s' \in \cR_i$, $J^\pi(s)=J^\pi(s')$ since $R^\pi_{i,\star}= \mathbf{1} g_i^\top$ where $g_i$ is stationary distribution of $R^\pi_{i}$. Thus,   $PJ^\pi(s,a)= P J^\pi(s,a) = \sum_{s'' \in \cR_{i}} P(s'' \,|\, s,a ) J^\pi (s'') = J^\pi(s)=J^\pi(s')$ for any $a \in \cA$.  
\end{proof}
Now we prove  Lemma~\ref{lem:pdl}. 
\begin{proof}
First, we have  $J^\pi-J^{\pi'}  =P^{\pi}_\star (J^{\pi} -J^{\pi'}) + (P^{\pi}_\star-I) J^{\pi'}$ since $P^\pi_\star P^\pi_\star=P^\pi_\star$. 

For the first term, 
\begin{align*}
    P^{\pi}_\star (J^{\pi} -J^{\pi'})&=    P^\pi_{\star} (r^\pi- J^{\pi'})\\&=    P^\pi_{\star} (r^\pi+V^{\pi'}-P^{\pi'} V^{\pi'} - r^{\pi'})\\&=    P^\pi_{\star} (r^\pi+P^\pi V^{\pi'}-P^{\pi'} V^{\pi'} - r^{\pi'})\\&=    P^\pi_{\star} (\Theta_\pi-\Theta_{\pi'})(r+P V^{\pi'})\\&=    P^\pi_{\star} (\Theta_\pi-\Theta_{\pi'})(PJ^{\pi'}+Q^{\pi'})
\end{align*}
where first and third equalities come from property of limiting matrix, $P^\pi_\star=P^\pi_\star P^\pi_\star= P^\pi_\star P^\pi$, second and last equalities are from Bellman equation, and fourth equality comes from that $\Theta_\pi \in  \real^{|\cS| \times |\cS||\cA|} $ is matrix form of policy $\pi$ satisfying $\Theta_\pi P = P^{\pi}$ and $\Theta_\pi r = r^{\pi}$. Then, for given $\mu$ with full support and $s \in \mathcal{R}$, we have
\begin{align*}
    \mu^\top P^{\pi}_\star (J^{\pi} -J^{\pi'})&=   \mu^\top P^\pi_{\star} (\Theta_\pi-\Theta_{\pi'})(PJ^{\pi'}+Q^{\pi'})\\& =\sum_{s \in \cS} d^\pi_{\mu}(s)\sum_{a \in \cA} (\pi(a \,|\, s)-\pi'(a \,|\, s))(PJ^{\pi'}(s,a)+Q^{\pi'}(s,a))\\& =\sum_{s \in \cR} d^\pi_{\mu}(s)\sum_{a \in \cA} (\pi(a \,|\, s)-\pi'(a \,|\, s))(PJ^{\pi'}(s,a)+Q^{\pi'}(s,a))\\& =\sum_{s \in \cR} d^\pi_{\mu}(s) \sum_{a \in \cA}(\pi(a \,|\, s)-\pi'(a \,|\, s))Q^{\pi'}(s,a)
\end{align*}
where third equality comes from $d^\pi_{\mu}(s)=0$ for all $s\in \cT$ and last equality is from Lemma~\ref{lem:con}.

For the second term, if  $s \in \mathcal{R}_i$, $P^{\pi}_\star J^{\pi'} (s)=\sum_{s' \in \cR_i} g_i^\pi(s') J^{\pi'} (s')=J^{\pi'} (s) $ since
$R^\pi_{i,\star}= \mathbf{1} (g^\pi_i)^\top$ where $g^\pi_i$ is stationary distribution, $J^{\pi'}(s')=J^{\pi'}(s'')$ for $s',s'' \in \cR_i$, and $\pi,\pi'$ share same recurrent class by Lemma~\ref{lem:con} and Fact~\ref{fact:classification}. Otherwise, if $s \in \cT$, define $J^\pi_{i} \in \real^{|\cR_{i}|}$ and $J^\pi_{m+1}\in \real^{|\cT|}$ for $1\le i \le m$ such that $J^\pi(s)=J^\pi_{i}(s)$ for $s \in \cR_i$ and $J^\pi(s)=J^\pi_{m+1}(s)$ for $s \in \cT$. Since $S^\pi_{i,\star} = (I-T^\pi)^{-1}S^\pi_i R^{\pi}_{i,\star}$, we have
\begin{align*}
 P^{\pi}_\star J^{\pi'}(s) &= \sum^m_{i=1} (I-T^\pi)^{-1}S^\pi_i J^{\pi'}_i (s)\\&=  (I-\bar{T}^\pi)^{-1}(P^\pi-\bar{T}^\pi) J^{\pi'}(s)    
\end{align*}
where first equality comes from Lemma~\ref{lem:con}, Fact~\ref{fact:classification}, and canonical form of $P^\pi_\star$ and $P^\pi$. This implies 
\begin{align*}
(P^{\pi}_\star-I)J^{\pi'}(s)&= (I-\bar{T}^\pi)^{-1}(P^\pi-I) J^{\pi'}(s)\\&=(I-\bar{T}^\pi)^{-1}(P^\pi-P^{\pi'}) J^{\pi'}(s)\\&=(I-\bar{T}^\pi)^{-1}(\Theta_\pi-\Theta_{\pi'}) PJ^{\pi'}(s)
\end{align*}
for $ s\in \cT$, where second equality is from Bellman equation. By multiplying $\mu$, we
conclude the proof. 
\end{proof}




\subsection{Proof of Theorem~\ref{thm::policy_gd}}
\begin{proof}
First, due to invariance of transient and recurrent class, it's sufficient to consider differentiability of $\{R^\pi_{\star, i}\}^m_{i=1}, \{S^\pi_{\star,i}\}^m_{i=1}$ since other entries of $P^{\pi }_\star $ are always zero.
The differentiability of $R^\pi_{\star, i}$ is guaranteed by the proof of \cite[Lemma~1]{marbach2001simulation} (As a technical detail, differentiability of $R^\pi_{\star,i}$ in \cite{marbach2001simulation} is proved under aperiodicity, but the proof also works for the unichain MDP due to uniqueness of stationary distribution \cite[Proposition 1.29]{levin2017markov}.), and differentiability of $(I-T^\pi)^{-1}$ is guaranteed by proof of \cite[Theorem~4.7]{lee2025policy}. These imply differentiability of $S^\pi_{i,\star} = (I-T^\pi)^{-1}S^\pi_i R^{\pi}_{i,\star}$ and $P^{\pi }_\star $  as well.

For manifold $S$, we denote tangent space of $S$ as $\mathbf{T} S$ \citep{lee2003smooth}. Consider $\Theta$ a differentiable manifold and differential $\nabla_\theta \pi_\theta: \mathbf{T} \Theta \rightarrow \mathbf{T} \Pi_+$  for $\pi_\theta: \Theta \rightarrow \Pi_+$, where $ \mathbf{T} \Pi_+ =\{ u\in \real^{|\cS||\cA|} \,|\,\sum_{a\in\cA} u(a \,|\,s ) =0\, \forall s \in \cS\}$. (Here note that $\Pi_+$ is simplex embedded in Euclidean space.)  Since 
$J^\pi_\mu:  \Pi_+ \rightarrow  \real$ is smooth function by previous argument, we can also consider differential $\nabla _\pi J^\pi_\mu: \mathbf{T} \Pi_+ \rightarrow \mathbf{T} \real$ where $\mathbf{T} \real =\real$. 

By Lemma~\ref{lem:pdl}, for any $u \in \mathbf{T} \Pi_+$ and $\epsilon >0$ such that $\pi+\epsilon  u \in \Pi_+$, we have
\begin{align*}
&J_\mu^{\pi+\epsilon  u} - J_\mu^{\pi}=        \epsilon  \sum_{\substack{s \in \cR\\a \in \cA}}
 d^{\pi+\epsilon  u}_{\mu}(s) u(a \mid s)Q^{\pi}(s,a) +\epsilon 
\sum_{\substack{s \in \cT\\a \in \cA}}\delta^{\pi+\epsilon  u}_{\mu}(s)u(a \mid s) PJ^{\pi}(s,a)
\end{align*}
By continuity of $d_\mu^{\pi}$ and $\delta^{\pi}_{\mu}$ on $\Pi_+$ and dividing both sides by $\epsilon$ and letting $\epsilon \rightarrow 0$, the differential on $\mathbf{T}\Pi_+$ can be represented by
\[\nabla _\pi J^\pi_\mu (s ,a) = d^{\pi}_{\mu}(s) Q^{\pi}(s,a)\mathbf{1}_{\cR}(s) +
\delta^{\pi}_{\mu}(s) PJ^{\pi}(s,a)\mathbf{1}_{\cT}(s).\]
This representative is defined up to an action-independent constant at each state, since tangent vectors sum to zero over actions.
Then, by chain rule,
\[\nabla _\theta J^\pi_\mu=\sum_{\substack{s \in \cR\\a \in \cA}} d^{\pi}_{\mu}(s) \nabla _\theta \pi (a \,|\, s) Q^{\pi}(s,a) +
\sum_{\substack{s \in \cT\\a \in \cA}} \delta^{\pi}_{\mu}(s)  \nabla _\theta \pi(a \,|\, s) PJ^{\pi}(s,a). \]

\end{proof}

We briefly note that Lemma~\ref{lem:continuity} directly follows from Theorem~\ref{thm::policy_gd}. 


\section{Omitted proofs in Section \ref{sec:3}}\label{appen:miss_3}

We first show how $\alpha$-policy gradient algorithms can be derived from the mirror descent framework.

Let \(h:\cM(\cA)\to\mathbb{R}\) be a strictly convex function and continuously
differentiable on the relative interior of \(\cM(\cA)\), denoted as
\(\rint\cM(\cA)\).
Define the Bregman divergence generated by \(h\) as
\[
D(p,p') \;=\; h(p)-h(p')-\langle \nabla h(p'),\, p-p' \rangle,
\qquad\forall\,p\in\cM(\cA),\,p'\in\rint\,\cM(\cA).
\]
In particular, if $h(p) =\frac{1}{2}\|p\|_2^2$,  then $D (p,p')= \frac{1}{2}\|p-p'\|^2_2$. If $h(p)=\sum_{a\in \cA}p(a) \log p(a)  $, then $D(p,p')= \sum_{a\in \cA}p(a) \log \frac{p(a)}{p'(a)}$. For $\rho \in \real_+^{|\cS|}$, define the weighted divergence
\[
D_\rho(\pi,\pi') \;=\; \sum_{s\in\cS}\rho(s)\,D(\pi(\cdot\,|\,s),\pi'(\cdot\,|\,s)).
\]
For policy $\pi$, define
\[
\rho_\mu^\pi(s)=
d^{\pi}_{\mu}(s)\mathbf{1}_{\mathcal{R}}(s)+\delta^{\pi}_{\mu}(s) \mathbf{1}_{\mathcal{T}}(s),
\]
where $\mathbf{1}_{\mathcal{R}}$ and $\mathbf{1}_{\mathcal{T}}$ are indicator functions.
Following the derivations of \cite{shani2020adaptive} and \cite{xiao2022convergence}, we consider
policy mirror ascent methods with weighted Bregman divergences:
\[
\pi_{k+1}
=\argmax_{\pi\in \mathcal{C}}
\left\{
\eta_k  \sum_{s \in \cS}
\sum_{a \in \cA} \nabla J_\mu^{\pi_{k}}(s,a)\pi(a \,|\, s)
-\,D_{\rho_\mu^{\pi_k}}\!\big(\pi,\pi_{k}\big)
\right\},
\]
where \(\eta_k\) is the step size, \(\mu\in\cM_+(\cS)\) is an initial state distribution with full support, and $\cC$ is compact subset of $\Pi_+$. 
For evaluating $\nabla J_\mu^{\pi_k}$, 
Theorem~\ref{thm::policy_gd} implicitly implies
\[\nabla J_\mu^{\pi}(s,a)
=
d^{\pi}_{\mu}(s)  Q^{\pi}(s,a)\mathbf{1}_{\mathcal{R}}(s)+\delta^{\pi}_{\mu}(s) PJ^{\pi}(s,a)\mathbf{1}_{\mathcal{T}}(s).
\]
Let $G^\pi(s,\cdot)= Q^{\pi}(s,\cdot) \mathbf{1}_{ \cR}(s) + PJ^\pi(s,\cdot)\mathbf{1}_{\cT} (s)$ for all $s\in \cS$. Then, for $\mu$ with full support on $\mathcal{S}$, since $d^{\pi}_{\mu}(s)>0$, $\delta^{\pi}_{\mu}(s')>0$ for any $s\in \cR$ and $s'\in \cT$, the update for $\pi_{k+1}$ splits across states: 
\[
\pi_{k+1}(\cdot \, |\, s)
=
\argmax_{p\in \mathcal{C}'}
\Big\{
\ \sum_{a\in \cA}\eta_kG^{\pi_k}(s,a)p(a)
- D\big(p(\cdot),\,\pi_{k}(\cdot \,|\, s)\big)
\Big\}, \qquad \forall s \in \cS.
\]

\subsection{Proof of Lemma~\ref{lem:bound}}
\begin{proof}
As we showed in proof of Theorem~\ref{thm::policy_gd}, $d^{\pi}$ and $\delta^{\pi}$ are continuous on $\Pi_+$. For each $i$, $\cR_i$ is single recurrent class on $\pi \in \Pi_\alpha$, so stationary distribution $g^\pi_i$ of $R^\pi_i$ is always positive  \cite[Proposition 1.19]{levin2017markov}. This implies $\min_{s \in \cR, \pi \in\Pi_\alpha} d^\pi_\mu(s) >0$  by continuity, and  $\|d^\pi_\mu\|_1=1$. On the other hand, for $s\in \cT$ , $\min_{\pi \in\Pi_\alpha} \delta^\pi_\mu(s) \ge \mu(s)$ by definition  and $\max_{\pi \in\Pi_\alpha} \delta^\pi_\mu(s)  <\infty$ by continuity. 
\end{proof}

\subsection{Proof of Lemma \ref{lem:ngd_descent}}
We prove a more detailed version of  Lemma \ref{lem:ngd_descent}. 

\begin{lemma}\label{lem:7'}
For $\pi_0\in\Pi_\alpha$ and $\mu \in \cM_+(\cS)$, the $\alpha$-policy mirror ascent generates a sequence of policies $\{\pi_{k}\}^\infty_{k=1}$ satisfying
\[ \sum_{a\in \cA}G^{\pi_{k}}(a, s )  (\pi_{k}(a\,|\,s)-\pi_{k+1}(a\,|\,s))  \le 0,\qquad \forall\, s\in \cS ,\]and
\[
J_\mu^{\pi_k} \le J_\mu^{\pi_{k+1}}.
\]
\end{lemma}

\begin{proof}
    Applying Fact \ref{fact:bregman} with
$\mathcal{C}= \cM_
\alpha(\cA)$, $\phi(p)=-\eta_k \sum_{a\in\cA} G^{\pi_k}(a,s)p(a)$, we obtain
\begin{align*}
    &\eta_k \sum_{a\in \cA} G^{\pi_k}(a,s) p(a)
+D\big(\pi_{k+1}(\cdot\,|\,s),\, \pi_k(\cdot\,|\,s)\big)
\\&\le
\eta_k \sum_{a\in \cA} G^{\pi_k}(a,s) \pi_{k+1}(a\,|\,s)
+ D\big(p, \pi_k(\cdot\,|\,s)\big) -D\big(p, \pi_{k+1}(\cdot\,|\,s)\big)
\end{align*}
for any $p\in\cM_\alpha(\cA)$. Rearranging terms and dividing both sides by $\eta_k$, we get
\begin{align*}
    &\sum_{a\in \cA} G^{\pi_k}(a,s)( p(a)-\pi_{k+1}(a\,|\,s))
+ \frac{1}{\eta_k} D\!\big(\pi_{k+1}(\cdot\,|\,s), \pi_k(\cdot\,|\,s)\big)
\\&\le\;
\frac{1}{\eta_k} D\!\big(p, \pi_k(\cdot\,|\,s)\big)
- \frac{1}{\eta_k} D\!\big(p, \pi_{k+1}(\cdot\,|\,s)\big) \tag{$*$}.
\end{align*}
Letting \(p=\pi_k(\cdot\,|\,s)\) in previous inequality yields
\[
\sum_{a\in \cA} G^{\pi_k}(a,s)(\pi_{k}(a\,|\,s) - \pi_{k+1}(a\,|\,s)) \le
-\frac{1}{\eta_k} D\big(\pi_{k+1}(\cdot\,|\,s), \pi_k(\cdot\,|\,s)\big)
-\frac{1}{\eta_k} D\big(\pi_k(\cdot\,|\,s), \pi_{k+1}(\cdot\,|\,s)\big).
\]
Then, the first result comes from nonnegativity of Bregman divergence and second result comes from performance difference lemma by multiplying $\rho_\mu^{\pi_{k+1}}(\cdot)$ both sides.
\end{proof}


\subsection{Proof of Theorem \ref{thm:ngd}}
\begin{proof}
Consider previous inequality ($*$). Let \(p=\pi_\alpha(\cdot \,|\, s)\in\cM_\alpha\)  and add--subtract \(\pi_{k}(\cdot \,|\, s)\) inside
the inner product. Then we have
\begin{align*}
    &\sum_{a\in \cA} G^{\pi_k}(a,s)(\pi_{k}(a\,|\,s) - \pi_{k+1}(a\,|\,s))
+ \sum_{a\in \cA} G^{\pi_k}(a,s)(\pi_\alpha(a\,|\,s) - \pi_{k}(a\,|\,s)) \\&\le\; \frac{1}{\eta_k} D\big(\pi_\alpha (\cdot\,|\,s),\pi_{k}(\cdot \,|\, s)\big)
        - \frac{1}{\eta_k} D\big(\pi_\alpha (\cdot\,|\,s),\pi_{k+1} (\cdot\,|\,s)\big).
\end{align*}
This implies
\begin{align*}
    &\sum_{s \in \cS}\rho_\mu^{\pi_\alpha}(s)\sum_{a\in \cA} G^{\pi_k}(a,s)(\pi_{k}(a\,|\,s) - \pi_{k+1}(a\,|\,s))
+ \sum_{s \in \cS}\rho_\mu^{\pi_\alpha}(s)\sum_{a\in \cA} G^{\pi_k}(a,s)(\pi_\alpha(a\,|\,s) - \pi_{k}(a\,|\,s))\\ & \le\; \frac{1}{\eta_k} D_{\rho^{\pi_\alpha}_\mu}(\pi_\alpha, \pi_k) - \frac{1}{\eta_k} D_{\rho^{\pi_\alpha}_\mu}(\pi_\alpha, \pi_{k+1}) .
\end{align*}

For the first term, 
\[
\begin{aligned}
&\sum_{s \in \cS}\rho^{\pi_\alpha}_\mu(s)\sum_{a\in \cA} G^{\pi_k}(a,s)(\pi_{k}(a\,|\,s) - \pi_{k+1}(a\,|\,s))
 \\
&= \sum_{s \in \cS}\rho^{\pi_{k+1}}_\mu(s) \frac{\rho^{\pi_\alpha}_\mu(s)}{\rho^{\pi_{k+1}}_\mu(s)}\sum_{a\in \cA} G^{\pi_k}(a,s)(\pi_{k}(a\,|\,s) - \pi_{k+1}(a\,|\,s)) \\
&\ge C_\alpha (J_{\mu}^{\pi_{k}}-J_{\mu}^{\pi_{k+1}}),
\end{aligned}
\]
where the last inequality comes from Lemma~\ref{lem:7'} and~\ref{lem:pdl}.

For the second term, by  Lemma~\ref{lem:pdl},
\[
\sum_{s \in \cS}\rho_\mu^{\pi_\alpha}(s)\sum_{a\in \cA} G^{\pi_k}(a,s)(\pi_\alpha(a\,|\,s) - \pi_{k}(a\,|\,s))
= J_\mu^{\pi_\alpha}-J_\mu^{\pi_{k}}  .
\]
Thus we have
\[
J_\mu^{\pi_\alpha}-J_\mu^{\pi_{k}} 
\;\le\;\frac{1}{\eta_k} D_{\rho_\mu^{\pi_\alpha}}(\pi_\alpha, \pi_k) - \frac{1}{\eta_k} D_{\rho_\mu^{\pi_\alpha}}(\pi_\alpha, \pi_{k+1})
+ C_\alpha(J_{\mu}^{\pi_{k+1}}
 - J_{\mu}^{\pi_{k}}).
\]

Setting \(\eta_k=\eta\) for all \(k\ge 0\) and summing over \(k\) gives
\[
\sum_{i=0}^{k}\big(J_\mu^{\pi_\alpha}- J_\mu^{\pi_{i}}  \big)
\;\le\;
\frac{1}{\eta} D_{\rho_\mu^{\pi_\alpha}}(\pi_\alpha, \pi_0)  - \frac{1}{\eta} D_{\rho_\mu^{\pi_\alpha}}(\pi_\alpha, \pi_{k+1}) 
+  C_\alpha(J_{\mu}^{\pi_{k+1}}
 - J_{\mu}^{\pi_{0}}).
\]

Since \(J_\mu^{\pi_{k}}\) are non-decreasing by Lemma \ref{lem:ngd_descent} and Bregman divergence is non-negative, we conclude that
\[
J_\mu^{\pi_\alpha}-J_\mu^{\pi_{k}}  
\;\le\; \frac{1}{k+1}\!\left( \frac{D_{\rho_\mu^{\pi_\alpha}}(\pi_\alpha, \pi_0)}{\eta}
+ C_
\alpha (J^{\pi_\alpha}_{\mu}- J_{\mu}^{\pi_{0}}) \right).
\]
\end{proof}


\subsection{Proof of Theorem \ref{thm:ngd_linear}}
Define
$
U^{\pi_\alpha}_k=J_\mu^{\pi_\alpha}- J_\mu^{\pi_{k}},
$
We first prove the following key lemma. 
\begin{lemma}\label{lem:ngd_linear}
For $\pi_0 \in \Pi_\alpha$ and a given $\mu \in \cM_+(\cS)$, the $\alpha$-clipped policy mirror ascent with step size $\eta_k>0$ generates a sequence of policies $\{\pi_{k}\}^\infty_{k=1}$ satisfying,
\[
C_\alpha\,\big(U^{\pi_\alpha}_{k+1}-U^{\pi_\alpha}_k\big) + U^{\pi_\alpha}_k
\;\le\;
\frac{1}{\eta_k} D_{\rho_\mu^{\pi_\alpha}}({\pi_\alpha}, \pi_k) - \frac{1}{\eta_k} D_{\rho_\mu^{\pi_\alpha}}({\pi_\alpha}, \pi_{k+1}) 
\]
\end{lemma}


\begin{proof}
In the previous proof of Theorem \ref{thm:ngd}, we showed that 
\begin{align*}
    &\sum_{s \in \cS}\rho_\mu^{\pi_\alpha}(s)\sum_{a\in \cA} G^{\pi_k}(a,s)(\pi_{k}(a\,|\,s) - \pi_{k+1}(a\,|\,s))
+ \sum_{s \in \cS}\rho_\mu^{\pi_\alpha}(s)\sum_{a\in \cA} G^{\pi_k}(a,s)(\pi_\alpha(a\,|\,s) - \pi_{k}(a\,|\,s))\\ & \le\; \frac{1}{\eta_k} D_{\rho_\mu^{\pi_\alpha}}(\pi_\alpha, \pi_k) - \frac{1}{\eta_k} D_{\rho_\mu^{\pi_\alpha}}(\pi_\alpha, \pi_{k+1}) .
\end{align*}
For the first term
\[
\begin{aligned}
&\sum_{s \in \cS}\rho_\mu^{\pi_\alpha}(s)\sum_{a\in \cA} G^{\pi_k}(a,s)(\pi_{k}(a\,|\,s) - \pi_{k+1}(a\,|\,s))\\
&\ge  C_\alpha (J_{\mu}^{\pi_{k}}-J_{\mu}^{\pi_{k+1}})
\end{aligned}
\]
where the inequality comes from Lemma~\ref{lem:7'} and~\ref{lem:pdl}.

For the second term, by Lemma~\ref{lem:pdl},
\[
\sum_{s \in \cS}\rho_\mu^{\pi_\alpha}(s)\sum_{a\in \cA} G^{\pi_k}(a,s)(\pi_\alpha(a\,|\,s) - \pi_{k}(a\,|\,s))
= J_\mu^{\pi_\alpha}- J_\mu^{\pi_{k}}  .
\]
We obtain desired result after substitution.

\end{proof}

We are now ready to prove Theorem \ref{thm:ngd_linear}
\begin{proof}
By Lemma \ref{lem:ngd_linear}, 
\[
C_\alpha\,\big(U^{\pi_\alpha}_{k+1}-U^{\pi_\alpha}_k\big) + U^{\pi_\alpha}_k
\;\le\;
\frac{1}{\eta_k} D_{\rho_\mu^{\pi_\alpha}}(\pi_\alpha, \pi_k) - \frac{1}{\eta_k} D_{\rho_\mu^{\pi_\alpha}}(\pi_\alpha, \pi_{k+1}). 
\]
Dividing both sides by $C_\alpha$ and rearranging terms, we obtain
\[
U^{\pi_\alpha}_{k+1}
+\frac{1}{\eta_k C_\alpha} D_{\rho_\mu^{\pi_\alpha}}(\pi_\alpha, \pi_{k+1}) 
\;\le\;
\left(1-\frac{1}{C_\alpha}\right)
\left(
U^{\pi_\alpha}_k + \frac{1}{\eta_k(C_\alpha-1)} D_{\rho_\mu^{\pi_\alpha}}({\pi_\alpha}, \pi_{k}) 
\right).
\]
Since the step sizes satisfy condition,
$\eta_{k+1}(C_\alpha-1)\ge \eta_kC_\alpha >0$, we have
\[
U^{\pi_\alpha}_{k+1}
+\frac{1}{\eta_{k+1}(C_\alpha-1)} D_{\rho_\mu^{\pi_\alpha}}({\pi_\alpha}, \pi_{k+1})
\;\le\;
\left(1-\frac{1}{C_\alpha}\right)
\left(
U^{\pi_\alpha}_k + \frac{1}{\eta_k(C_\alpha-1)} D_{\rho_\mu^{\pi_\alpha}}({\pi_\alpha}, \pi_{k})
\right).
\]
Therefore, by recursion,
\[
U^{\pi_\alpha}_k
+\frac{1}{\eta_k(C_\alpha-1)} D_{\rho_\mu^{\pi_\alpha}}({\pi_\alpha}, \pi_{k})
\;\le\;
\left(1-\frac{1}{C_\alpha}\right)^{\!k}
\left(
U^{\pi_\alpha}_0 + \frac{1}{\eta_0(C_\alpha-1)} D_{\rho_\mu^{\pi_\alpha}}({\pi_\alpha}, \pi_{0})
\right).
\]
\end{proof}

\section{Omitted proofs in Section~\ref{sec:4}}
\subsection{Proof of Theorem~\ref{thm:PE}}
Define  
\[\qquad V^{k+1} = P^\pi V^k+r^\pi \qquad k=0,1,\dots\]
where $V^0 =0$.
\begin{fact}[Classical result, \protect{\cite[Theorem~9.4.1]{10.5555/528623}}]\label{fact::vi_normalized_iterate}
Consider a general (multichain) MDP. Then, for $k \ge 1$, the sequence $\{V^k\}^\infty_{k=0}$ exhibits the rate
\[\infn{\frac{V^k}{k}-J^{\pi}} \le \frac{2\infn{V^{\pi}}}{k}.\]
\end{fact}

Using Fact~\ref{fact::vi_normalized_iterate}, we first study sample complexity for approximation of $PJ^\pi$.
\begin{lemma}\label{lem:approxk} 
    Let $\epsilon\in(0,1)$ and $ \delta\in(0,1)$. For given $\pi \in \Pi$, with $1-\delta$ probability, $\widehat{PJ}^\pi$ in critic satisfies $\infn{\widehat{PJ}^\pi - PJ^\pi} \le \epsilon$ with sample complexity 
    \[NH|\cS||\cA|=\mathcal{O}\left( \frac{ (1+\infn{V^\pi}+R)(1+R)^2|\cS||\cA|}{ \epsilon^3}\log\left(\frac{2|\cS||\cA|}{\delta}\right)\right).\]
\end{lemma}
\begin{proof}
Since
\begin{align*}
\widehat{PJ}(s_0,a_0)=\frac{1}{N}\sum^N_{j=1} \frac{1}{H+1}\sum^{H}_{i=0}r_j(s_i,a_i),  
\end{align*}
and  
\begin{align*}
   \expec_\pi \left[\frac{1}{N}\sum^N_{j=1} \frac{1}{H+1}\sum^{H}_{i=0}r_j(s_i,a_i)\right] =\expec_\pi \left[\frac{1}{H+1}\sum^{H}_{i=0}r_j(s_i,a_i)\right]
\end{align*}
by independence of samples from generative model. Also, we have
\begin{align*}
    \left|\expec_\pi \left[\frac{1}{H+1}\sum^{H}_{i=0}r_j(s_i,a_i)\right]-PJ^\pi(s_0,a_0) \right|&= \left|\frac{1}{H+1}(PV^H+r)(s_0,a_0) -P J^\pi(s_0,a_0) \right| 
    \\&=\left|P\left(\frac{V^{H}}{H+1}-J^\pi\right)(s_0,a_0) +\frac{r(s_0,a_0)}{H+1}\right|\\&=
    \left|\frac{H}{H+1}P\left(\frac{V^{H}}{H}-J^\pi\right)(s_0,a_0) +\frac{(r-K)(s_0,a_0)}{H+1}\right|
    \\&
    \le \frac{2(\infn{V^\pi}+R)}{H+1}    
\end{align*}
where first equality comes from Bellman equation and last inequality follows from Fact \ref{fact::vi_normalized_iterate} and $\infn{K} \le R$. Moreover, by Hoeffding inequality with $N=\max\{1,\lceil \frac{2R^2}{ (\epsilon')^2}\log(\frac{2}{\delta})\rceil\}$, we get
\[\text{Prob} \left(\left|\widehat{PJ}^\pi(s_0,a_0) - \expec_\pi \left[\frac{1}{N}\sum^N_{j=1} \frac{1}{H+1}\sum^{H}_{i=0}r_j(s_i,a_i)\right]\right|  \le \epsilon' \right) \ge 1-\delta.\]
and by union bound over $(s_0,a_0) \in \cS \times \cA$, we have
$\infn{\widehat{PJ}^\pi-\expec \left[ \widehat{PJ}^\pi \right]}< \epsilon'$.
Let $H =\max\{1,\lceil\frac{4(\infn{V^\pi}+R)}{\epsilon}\rceil\}$ and $\epsilon'=\epsilon/2$. Then, by triangular inequality $\infn{PJ^\pi-\widehat{PJ}^\pi} \le \infn{PJ^\pi-\expec \left[\widehat{PJ}^\pi\right]} +\infn{\expec \left[\widehat{PJ}^\pi\right]-\widehat{PJ}^\pi} \le \epsilon$ with sample complexity of  
\[NH |\cS||\cA|=\mathcal{O}\left(\frac{|\cS||\cA| (1+\infn{V^\pi}+R)(1+R)^2}{\epsilon^3}\log\left(\frac{2|\cS||\cA|}{\delta}\right)\right).\]
\end{proof}
Now, we establish following lemma for evaluation of $Q^\pi$ with critic method.
\begin{lemma}\label{lem:tar} Consider a general (multichain) MDP. Then, for any recurrent state $s \in \mathcal{R}$ for $k \ge 1$, $\{V^k\}^\infty_{k=0}$ exhibits the rate
\[\left|\frac{1}{k}\sum^k_{i=1} (V^i-i J^\pi)(s) -V^\pi(s) \right| \le \frac{2 t_{\mathrm{tar}} \infn{V^\pi}}{k}.\]
\end{lemma}
\begin{proof} First, 
    \begin{align*}
        V^k &= \sum^{k-1}_{i=0}(\cP^\pi)^ir^\pi
        \\&= \sum^{k-1}_{i=0}(\cP^\pi)^i(J^\pi+(I-P^\pi)V^\pi)
        \\&= k J^\pi+V^\pi-(\cP^\pi)^{k} V^\pi.
    \end{align*}
 Thus 
    \begin{align*}
        \frac{1}{k}\sum^k_{i=1} (V^i-iJ^\pi)&=V^\pi-\frac{1}{k}\sum^{k}_{i=1}(\cP^\pi)^i V^\pi\\&=V^\pi+\left(P^\pi_\star-\frac{1}{k}\sum^{k}_{i=1}(\cP^\pi)^i\right)V^\pi
    \end{align*}
    where first equality is from $V^0=0$ and last equality comes from $P^\pi_\star V^\pi=0$ \cite[Section A.5]{10.5555/528623}. Then by canonical form of stochastic matrix and Fact \ref{fact:hittting}, we have for $s\in \cR$,  
    \[\left\|P^\pi_\star(s,\cdot) -\frac{1}{k}\sum^{k}_{i=1}(\cP^\pi)^i (s,\cdot)\right\|_1 \le    \frac{2t^\pi_{\mathrm{tar}} }{k} \qquad \forall s \in \mathcal{R} \]
    where $t^\pi_{\mathrm{tar}} = \max_{1\le j \le m} t^{\pi,j}_{\mathrm{tar}}$.
\end{proof}
Using previous lemma, we establish following sample complexity result. 
\begin{lemma}\label{lem:approxq} Let $\epsilon\in(0,1)$ and $ \delta\in(0,1)$. For given $\pi \in \Pi_+$, with $1-\delta$ probability 
\[\infn{\widehat{Q}^{\pi}(\cdot, \cdot\cdot)\mathbf{1}_{\cR}(\cdot) - Q^\pi(\cdot, \cdot\cdot) \mathbf{1}_{\cR}(\cdot)} \le \frac{(2t_{\mathrm{tar}}+1) \infn{V^\pi}}{H'+1}+\frac{(H'+2)\infn{PJ^\pi-\widehat{PJ}^\pi}}{2}+\epsilon. \] 
with sample complexity 
    \[N'H'|\cS||\cA|= \mathcal{O}\left( \frac{(1+R)^2H'(H'+2)^2|\cS||\cA|}{ \epsilon^2}\log(\frac{2|\cS||\cA|}{\delta})\right).\]
\end{lemma}
\begin{proof}
    $\widehat{Q}^{j}(s_0,a_0)=\frac{1}{H'+1}\sum_{h=0}^{H'} \sum^h_{i=0}\big(r_j(s_i,a_i)-\widehat{PJ}^{\pi}(s_i,a_i)\big)$.
First, we can decompose as follows
\begin{align*}
&\left| \frac{1}{N'}\sum^{N'}_{j=1}\widehat{Q}^{j}(s_0,a_0)-Q^{\pi}(s_0,a_0)\right|\\&\le      \frac{1}{N'}\sum^{N'}_{j=1}\left|\frac{1}{H'+1}\sum_{h=0}^{H'} \sum^h_{i=0}\big(r_j(s_i,a_i)-\widehat{PJ}^{\pi}(s_i,a_i)\big) - \frac{1}{H'+1}\sum_{h=0}^{H'} \sum^h_{i=0}\big(r_j(s_i,a_i)-PJ^{\pi}(s_i,a_i)\big)\right|\\&+ \left| \frac{1}{N'}\sum^{N'}_{j=1}\frac{1}{H'+1}\sum_{h=0}^{H'} \sum^h_{i=0}\big(r_j(s_i,a_i)-PJ^{\pi}(s_i,a_i)\big)- \expec_\pi \left[\frac{1}{H'+1}\sum_{h=0}^{H'} \sum^h_{i=0}\big(r_j(s_i,a_i)-PJ^{\pi}(s_i,a_i)\big)\right]\right|\\& + \frac{1}{N'}\sum^{N'}_{j=1}\left|\expec_\pi \left[\frac{1}{H'+1}\sum_{h=0}^{H'} \sum^h_{i=0}\big(r_j(s_i,a_i)-PJ^{\pi}(s_i,a_i)\big)\right]-Q^\pi(s_0,a_0)\right|
\end{align*}

For the first term, 
 \begin{align*}
    & \left|\frac{1}{H'+1}\sum_{h=0}^{H'} \sum^h_{i=0}\big(r_j(s_i,a_i)-\widehat{PJ}^{\pi}(s_i,a_i)\big) - \frac{1}{H'+1}\sum_{h=0}^{H'} \sum^h_{i=0}\big(r_j(s_i,a_i)-PJ^{\pi}(s_i,a_i)\big)\right|
    \\ &= \left|-\frac{1}{H'+1}\sum_{h=0}^{H'} \sum^h_{i=0}\widehat{PJ}^{\pi}(s_i,a_i) + \frac{1}{H'+1}\sum_{h=0}^{H'} \sum^h_{i=0} PJ^{\pi}(s_i,a_i)\right|\\&\le \frac{(H'+2)\infn{PJ^\pi-\widehat{PJ}^\pi}}{2}
 \end{align*}

 For the last term, 
\begin{align*}
    &\left|Q^\pi(s_0,a_0) - \expec_\pi \left[\frac{1}{H'+1}\sum_{h=0}^{H'} \sum^h_{i=0}\big(r_j(s_i,a_i)-PJ^{\pi}(s_i,a_i)\big)\right]\right|\\&=
     \left|(P V^\pi+r-PJ^\pi)(s_0,a_0)- \expec_\pi \left[\frac{1}{H'+1}\sum_{h=0}^{H'} \sum^h_{i=0}\big(r_j(s_i,a_i)-PJ^{\pi}(s_i,a_i)\big)\right]\right|\\&=
     \left|P V^\pi(s_0,a_0)- \expec_\pi \left[\frac{1}{H'+1}\sum_{h=1}^{H'} \sum^h_{i=1}\big(r_j(s_i,a_i)-PJ^{\pi}(s_i,a_i)\big)\right]\right|\\&=
     \left|P V^\pi(s_0,a_0) -  P\left(\frac{1}{H'+1}\sum^{H'}_{i=1} (V^i-i J^\pi) \right)(s_0,a_0)\right|\\&=\left|\frac{1}{H'+1}P V^\pi(s_0,a_0)-\frac{H'}{H'+1}P\left(\frac{1}{H'}\sum^{H'}_{i=1} (V^i-i J^\pi) -V^\pi \right)(s_0,a_0)\right|\\&\le  \frac{(2t^\pi_{\mathrm{tar}}+1) \infn{V^\pi}}{H'+1}
\end{align*}
where last inequality is from Lemma~\ref{lem:tar}.

For the second term, let $Q'_j=\frac{1}{H'+1}\sum_{h=0}^{H'} \sum^h_{i=0}\big(r_j(s_i,a_i)-PJ^{\pi}(s_i,a_i)\big)$ by Hoeffding inequality with $N'=\max\{1,\lceil\frac{2R^2(H'+2)^2}{\epsilon^2}\log(\frac{2}{\delta})\rceil\}$, 
\[\text{Prob}\left(\left| \frac{1}{N'}\sum^{N'}_{j=1}Q'_j(s_0,a_0)- \expec \left[\frac{1}{N'}\sum^{N'}_{j=1}\frac{1}{H'+1}\sum_{h=0}^{H'} \sum^h_{i=0}\big(r_j(s_i,a_i)-PJ^{\pi}(s_i,a_i)\big)\right]\right|  \le \epsilon \right) \ge 1-\delta.\]
By union bound over $(s,a) \in \cS \times \cA$, we conclude the result.
\end{proof}

We are now ready to prove Theorem~\ref{thm:PE}.

\begin{proof}
In Lemma~\ref{lem:approxq}, let
\[
H'=\max\left\{1,\left\lceil\frac{3(2t^\pi_{\mathrm{tar}}+1)\infn{V^\pi}}{\epsilon'}\right\rceil\right\},
\qquad \epsilon=\frac{\epsilon'}{3},\qquad \epsilon'\in(0,1).
\]
Use failure probability $\delta/2$ in each of Lemmas~\ref{lem:approxk} and \ref{lem:approxq}. Then the sample complexity for the second set of trajectories is
\begin{align*}
N'H'|\cS||\cA|
&=\mathcal{O}\left(\frac{(1+R)^2H'(H'+2)^2|\cS||\cA|}{(\epsilon'/3)^2}\log\left(\frac{4|\cS||\cA|}{\delta}\right)\right)\\
&=\mathcal{O}\left(\frac{(1+R)^2(1+t^\pi_{\mathrm{tar}})^3(1+\infn{Q^\pi})^3|\cS||\cA|}{(\epsilon')^5}\log\left(\frac{4|\cS||\cA|}{\delta}\right)\right),
\end{align*}
where we used $\infn{V^\pi}\le\infn{Q^\pi}$. Also, by Lemma~\ref{lem:approxk}, we obtain
$\infn{PJ^\pi-\widehat{PJ}^\pi}\le 2\epsilon'/[3(H'+2)]$ with sample complexity
\begin{align*}
NH|\cS||\cA|
&=\mathcal{O}\left(\frac{(1+\infn{V^\pi}+R)(1+R)^2|\cS||\cA|}{\left(\frac{2\epsilon'}{3(H'+2)}\right)^3}\log\left(\frac{4|\cS||\cA|}{\delta}\right)\right)\\
&=\mathcal{O}\left(\frac{(1+R)^3(1+t^\pi_{\mathrm{tar}})^3(1+\infn{Q^\pi})^4|\cS||\cA|}{(\epsilon')^6}\log\left(\frac{4|\cS||\cA|}{\delta}\right)\right).
\end{align*}
Consequently,
\[
\infn{\widehat{Q}^{\pi}(\cdot,\cdot\cdot)\mathbf{1}_{\cR}(\cdot)-Q^\pi(\cdot,\cdot\cdot)\mathbf{1}_{\cR}(\cdot)}
\le \frac{\epsilon'}{3}+\frac{\epsilon'}{3}+\frac{\epsilon'}{3}=\epsilon'.
\]
Combining the two estimates and applying a union bound, with probability at least $1-\delta$ we have $\infn{G^\pi-\widehat{G}^\pi}\le\epsilon'$ with total sample complexity
\[
\mathcal{O}\left(\frac{(1+R)^3(1+t^\pi_{\mathrm{tar}})^3(1+\infn{Q^\pi})^4|\cS||\cA|}{(\epsilon')^6}\log\left(\frac{2|\cS||\cA|}{\delta}\right)\right).
\]
\end{proof}

\subsection{Proof of Lemma~\ref{lem:class}}
Denote $\cF_t =\{ \sigma (\{s_i : 0\le i \le t \})\}$ as the natural filtration generated by the sampling process in Markov chain.

First we prove the following lemma. 
\begin{lemma}\label{lem:cl}
Let $s_0 \in \mathcal{R}_i$ and $\pi \in \Pi_+$. For given deterministic $t,n>0$, 
   \[ \text{Prob}((t^i_{\mathrm{cov}}-t)_+\ge n \,|\, \cF_t)  \le \frac{\expec [(t^i_{\mathrm{cov}}-t)_+ \,|\, \cF_t]}{n}\le \frac{t^{\pi,i}_{\mathrm{cov}}}{n}\]
\end{lemma} 
\begin{proof}
    First inequality is from Markov's inequality. For second inequality, define new cover time $t'_{\mathrm{cov}}$ start from $t$, i.e, $t'_{\mathrm{cov}}=  \inf\{u\ge 0: \{s_j\}^{u+t-1}_{j=t}=\mathcal{R}_i\}$ then $t^i_{\mathrm{cov}}-t \le t'_{\mathrm{cov}}$ almost surely by definition of cover time, and this implies $\expec [(t^i_{\mathrm{cov}}-t)_+ \,|\, \cF_t] \le \expec [t'_{\mathrm{cov}} \,|\, \cF_t]= \expec [t'_{\mathrm{cov}} \,|\, s_t] \le t^{\pi,i}_{\mathrm{cov}}$.
\end{proof}
From Lemma~\ref{lem:cl} and definition of $t^{\pi}_{\mathrm{cov}}$, we can directly obtain following corollary. 
\begin{corollary}\label{cor:cl}
Let $s_0 \in \mathcal{R}$ and $\pi \in \Pi_+$. For given deterministic $t,n>0$, 
   \[ \text{Prob}((t_{\mathrm{cov}}-t)_+\ge n \,|\, \cF_t)  \le \frac{\expec [(t_{\mathrm{cov}}-t)_+ \,|\, \cF_t]}{n}\le \frac{t^{\pi}_{\mathrm{cov}}}{n}\]
   where $t_{\mathrm{cov}}=  \inf\{u\ge 0: \{s_j\}^{u-1}_{j=0}=\mathcal{R}_i \text{\,\,for some $i$}\}$ 
\end{corollary}
Now, we prove Lemma~\ref{lem:class}.

\begin{proof}
Suppose $s_0$ is a transient state. By the definition of transient half-life and submultiplicativity,
\[
\infn{(\bar T^\pi)^{M_1}}
\le 2^{-\lceil\log_2(2/\delta)\rceil}\le\delta/2.
\]
Thus $s_{M_1}$ is recurrent with probability at least $1-\delta/2$. The same conclusion holds with probability $1$ if $s_0$ is recurrent.

For a trajectory starting in a recurrent class, set $b=\lceil3t^\pi_{\mathrm{cov}}\rceil$. Corollary~\ref{cor:cl} gives, for every integer $j\ge1$,
\[
\mathrm{Prob}\big((t_{\mathrm{cov}}-(j-1)b)_+\ge b\mid\cF_{(j-1)b}\big)
\le\frac{t^\pi_{\mathrm{cov}}}{b}\le\frac13\le e^{-1}.
\]
Induction over these integer blocks yields
$\mathrm{Prob}(t_{\mathrm{cov}}\ge jb)\le e^{-j}$.
With $j=\lceil\log(2/\delta)\rceil$, the probability of failing to cover the recurrent class within $M_2=jb$ steps is at most $\delta/2$.

Apply this bound conditionally at time $M_1$. With probability at least $1-\delta$, $s_{M_1}$ is recurrent and $\{s_j\}_{j=M_1}^{M_1+M_2-1}$ contains every state in its recurrent class. Since that class is closed, this set equals the class.
\end{proof}

\subsection{Inexact $\alpha$-clipped policy mirror ascent}
First, we consider \emph{inexact policy mirror ascent}
\begin{align*}
\pi_{k+1}(\cdot \,|\, s)
=
\argmax_{p \in \cM_\alpha(\mathcal{A})}
\left\{
\eta_k \sum_{a\in \cA} \widehat{G}^{\pi_k}(s,a) p(a) 
- D(p(\cdot),\pi_k(\cdot\,|\, s))
\right\},
\qquad \forall s \in \mathcal{S},
\end{align*}
where $\widehat{G}^{\pi_k}$ is an inexact evaluation of $G^{\pi_k}$.
We first study the convergence properties of inexact policy mirror ascent under the following
assumption.
\begin{assump}\label{assum:inexact}
The inexact evaluations $\widehat{G}^{\pi_k}$ satisfy
\begin{align*}
\left\|
\widehat{G}^{\pi_k} - G^{\pi_k}
\right\|_\infty
\le \epsilon,
\qquad \forall k \ge 0 .
\end{align*}
\end{assump}
We first prove the key Lemma, counterpart of Lemma \ref{lem:ngd_descent}.

\begin{lemma}\label{inexact:mon}
For given $\pi_0\in \Pi_\alpha$ and $\mu \in \cM_+(\cS)$, the inexact $\alpha$-clipped policy mirror ascent with step size $\eta_k>0$ generates a sequence of policies $\{\pi_{k}\}^\infty_{k=1}$ satisfying
\begin{align*}
\sum_{a\in\cA}\widehat{G}^{\pi_k}(s,a)(\pi_{k}(\cdot \,|\, s)-\pi_{k+1}(\cdot \,|\, s)) 
\le 0,
\qquad \forall s \in \mathcal{S},
\end{align*}
and
\begin{align*}
J_\mu^{\pi_k}-J_\mu^{\pi_{k+1}} 
\le
2  B_\alpha \epsilon .
\end{align*}
\end{lemma}
\begin{proof}
The first inequality directly follows from the same arguments as in Lemma \ref{lem:7'}. By Lemma~\ref{lem:pdl},
\begin{align*}
  J_\mu^{\pi_k}-J_\mu^{\pi_{k+1}} 
&=
\sum_{s \in \cS}\rho_\mu^{\pi_{k+1}}(s)\sum_{a\in \cA} G^{\pi_k}(a,s)(\pi_{k}(a\,|\,s) - \pi_{k+1}(a\,|\,s))
\\
&=
\sum_{s \in \cS}\rho_\mu^{\pi_{k+1}}(s)\sum_{a\in \cA} \widehat{G}^{\pi_k}(a,s)(\pi_{k}(a\,|\,s) - \pi_{k+1}(a\,|\,s))
\\
&\quad
+
\sum_{s \in \cS}\rho_\mu^{\pi_{k+1}}(s)\sum_{a\in \cA} (G^{\pi_k}(a,s)-\widehat{G}^{\pi_k}(a,s))(\pi_{k}(a\,|\,s) - \pi_{k+1}(a\,|\,s)).
\end{align*}

For the second term,  
\begin{align*}
\sum_{a\in \cA} (G^{\pi_k}(a,s)-\widehat{G}^{\pi_k}(a,s))(\pi_{k}(a\,|\,s) - \pi_{k+1}(a\,|\,s))
&\le
\left\|
G^{\pi_k} - \widehat{G}^{\pi_k}
\right\|_\infty
\left\|
\pi_{k}(a\,|\,s) - \pi_{k+1}(a\,|\,s)
\right\|_1
\nonumber\\
&\le
2 \left\|
\widehat{G}^{\pi_k} - G^{\pi_k}
\right\|_\infty
\nonumber\\
&\le
2\epsilon ,
\end{align*}
where the first inequality is from Holder's inequality and the last inequality follows from Assumption \ref{assum:inexact}. Since  $\max_{\pi \in \Pi_{\alpha}}\|\rho^\pi\|_1 \le B_\alpha$,  we get the desired result. 

\end{proof}

Following fact will be used in the proof of Theorem \ref{thm:inexact}. 
\begin{fact}\label{fact:seq}[\cite{xiao2022convergence}]
Suppose $0 < \alpha < 1$, $b > 0$, and a nonnegative sequence
$\{a_k\}$ satisfies
\begin{align*}
a_{k+1} \le \alpha a_k + b,
\qquad \forall k \ge 0 .
\end{align*}
Then for all $k \ge 0$,
\begin{align*}
a_k \le \alpha^k a_0 + \frac{b}{1 - \alpha}.
\end{align*}
\end{fact}

Now, we present the convergence result of inexact policy mirror ascent. 

\begin{theorem}\label{thm:inexact}
Consider a multichain MDP. Under Assumption~\ref{assum:inexact}, for $\pi_0 \in \Pi_\alpha$ and given $\mu \in \cM_+(\cS)$, the inexact $\alpha$-clipped policy mirror ascent with constant step size $\eta>0$ generates a sequence of policies $\{\pi_{k}\}^\infty_{k=1}$ satisfying
\[J_\mu^{\pi_\alpha}-J_\mu^{\pi_{k}}  
\;\le\; \frac{1}{k+1}\!\left( \frac{D_{\rho_\mu^{\pi_\alpha}}({\pi_\alpha}, \pi_0)}{\eta}
+ C_
\alpha (J^{\pi_\alpha}_{\mu}-J^{\pi_0}_{\mu}) \right)+(2C_\alpha+k+2)B_\alpha \epsilon.\]
and with  adaptive stepsize $\eta_{k+1}(C_\alpha - 1) \ge \eta_k C_\alpha>0$
generates a sequence of policies $\{\pi_{k}\}^\infty_{k=1}$ satisfying
\begin{equation}
J^{\pi_\alpha}-J_\mu^{\pi_k}
\le
\left(1 - \frac{1}{C_\alpha}\right)^k
\left(
J^{\pi_\alpha}_{\mu}-J^{\pi_0}_{\mu}
+
\frac{1}{\eta_0(C_\alpha - 1)} D_{\rho_\mu^{\pi_\alpha}}({\pi_\alpha}, \pi_0)
\right)
+
4C_\alpha B_\alpha  \epsilon.
\end{equation}
\end{theorem}
\begin{proof}
Following the same arguments as in proof of Theorem \ref{thm:ngd},
\begin{align*}
    &\sum_{s \in \cS}\rho_\mu^{\pi_\alpha}(s)\sum_{a\in \cA} \widehat{G}^{\pi_k}(a,s)(\pi_{k}(a\,|\,s) - \pi_{k+1}(a\,|\,s))
+ \sum_{s \in \cS}\rho_\mu^{\pi_\alpha}(s)\sum_{a\in \cA} \widehat{G}^{\pi_k}(a,s)(\pi_\alpha(a\,|\,s) - \pi_{k}(a\,|\,s))\\ & \le\; \frac{1}{\eta_k} D_{\rho^{\pi_\alpha}_\mu}(\pi_\alpha, \pi_k) - \frac{1}{\eta_k} D_{\rho^{\pi_\alpha}_\mu}(\pi_\alpha, \pi_{k+1}) .
\end{align*}

For the first term,
\begin{align*}
&\sum_{s \in \cS}\rho^{\pi_\alpha}_\mu(s)\sum_{a\in \cA} \widehat{G}^{\pi_k}(a,s)(\pi_{k}(a\,|\,s) - \pi_{k+1}(a\,|\,s))
\\& \ge C_\alpha \sum_{s \in \cS}\rho^{\pi_{k+1}}_\mu(s)\sum_{a\in \cA} \widehat{G}^{\pi_k}(a,s)(\pi_{k}(a\,|\,s) - \pi_{k+1}(a\,|\,s))
\\&=C_\alpha \sum_{s \in \cS}\rho^{\pi_{k+1}}_\mu(s)\sum_{a\in \cA} G^{\pi_k}(a,s)(\pi_{k}(a\,|\,s) - \pi_{k+1}(a\,|\,s))\\&+C_\alpha \sum_{s \in \cS}\rho^{\pi_{k+1}}_\mu(s) \sum_{a\in \cA} (\widehat{G}^{\pi_k}(a,s)-G^{\pi_k}(a,s))(\pi_k(a\,|\,s) - \pi_{k+1}(a\,|\,s)) 
 \\
&\ge C_\alpha \sum_{s \in \cS}\rho^{\pi_{k+1}}_\mu(s) \sum_{a\in \cA} G^{\pi_k}(a,s)(\pi_{k}(a\,|\,s) - \pi_{k+1}(a\,|\,s))-2C_\alpha B_\alpha \epsilon
\\&\ge C_\alpha (J_\mu^{\pi_{k}} - J_\mu^{\pi_{k+1}}) -  2C_\alpha B_\alpha \epsilon
\end{align*}
where second to the last inequality is from Lemma~\ref{inexact:mon} and Assumption~\ref{assum:inexact}, and last inequality comes from Lemma~\ref{lem:pdl}.

For the second term,
\begin{align*}
&\sum_{s \in \cS}\rho^{\pi_\alpha}_\mu(s)\sum_{a\in \cA} \widehat{G}^{\pi_k}(a,s)(\pi_\alpha(a\,|\,s) - \pi_{k}(a\,|\,s))
\\&=\sum_{s \in \cS}\rho^{\pi_\alpha}_\mu(s) \sum_{a\in \cA} G^{\pi_k}(a,s)(\pi_\alpha(a\,|\,s) - \pi_{k}(a\,|\,s)) \\&+  \sum_{s \in \cS}\rho^{\pi_\alpha}_\mu(s) \sum_{a\in \cA} (\widehat{G}^{\pi_k}(a,s)-G^{\pi_k}(a,s))(\pi_\alpha(a\,|\,s) - \pi_{k}(a\,|\,s))
\\&\ge J_\mu^{\pi_\alpha} - J_\mu^{\pi_k}   - 2  B_\alpha\epsilon.
\end{align*}
where last inequality comes from Lemma~\ref{lem:7'} and~ \ref{lem:pdl} and Assumption~\ref{assum:inexact}.

First, consider constant step size. We have
\[
J_\mu^{\pi_\alpha}-J_\mu^{\pi_{k}} 
\;\le\;\frac{1}{\eta_k} D_{\rho_\mu^{\pi_\alpha}}(\pi_\alpha, \pi_k) - \frac{1}{\eta_k} D_{\rho_\mu^{\pi_\alpha}}(\pi_\alpha, \pi_{k+1})
+ C_\alpha(J_{\mu}^{\pi_{k+1}}
 - J_{\mu}^{\pi_{k}})+2(C_\alpha+1)  B_\alpha\epsilon
\]
and summing over \(k\) gives
\[
\sum_{i=0}^{k}\big(J_\mu^{\pi_\alpha}- J_\mu^{\pi_{i}}  \big)
\;\le\;
\frac{1}{\eta} D_{\rho_\mu^{\pi_\alpha}}({\pi_\alpha}, \pi_0)  - \frac{1}{\eta} D_{\rho_\mu^{\pi_\alpha}}({\pi_\alpha}, \pi_{k+1}) 
+  C_\alpha(J_{\mu}^{\pi_{k+1}}
 - J_{\mu}^{\pi_{0}})+2(k+1)(C_\alpha+1)B_\alpha \epsilon
\]

By Lemma \ref{inexact:mon},  $-J_\mu^{\pi_{k}} 
-2  B_\alpha \epsilon \le -J_\mu^{\pi_{k-1}} $ and this implies
\[
J_\mu^{\pi_\alpha}-J_\mu^{\pi_{k}}  
\;\le\; \frac{1}{k+1}\!\left( \frac{D_{\rho_\mu^{\pi_\alpha}}({\pi_\alpha}, \pi_0)}{\eta}
+ C_
\alpha(J^{\pi_\alpha}_{\mu}-J^{\pi_0}_\mu) \right)+kB_\alpha\epsilon+2(C_\alpha+1)B_\alpha \epsilon.
\]

Second, consider adaptive step size.
\begin{equation*}
C_\alpha
\left(
U_{k+1} - U_k - 2B_\alpha \epsilon
\right)
+ U_k
\le
\frac{1}{\eta_k} D_k
-
\frac{1}{\eta_k} D_{k+1}
+
2B_\alpha \epsilon,
\end{equation*}
where $U_k = J_\mu^{\pi_\alpha}-J_\mu^{\pi_k} $ and $D_k = D_{\rho^{\pi_\alpha}_\mu}(\pi_\alpha, \pi_{k})$.
This implies
\begin{equation*}
C_\alpha(U_{k+1} - U_k) + U_k
\le
\frac{1}{\eta_k} D_k
-
\frac{1}{\eta_k} D_{k+1}
+
2B_\alpha( C_\alpha+1)\epsilon.
\end{equation*}

Dividing both sides by $C_\alpha$ and rearranging terms with adaptive stepsize yields 
\begin{equation*}
U_{k+1}
+
\frac{1}{\eta_{k+1} (C_\alpha-1)} D_{k+1}
\le
\left(1 - \frac{1}{C_\alpha}\right)
\left(
U_k
+
\frac{1}{\eta_k(C_\alpha - 1)} D_k
\right)
+
4B_\alpha \epsilon.
\end{equation*}
Finally, Fact \ref{fact:seq} with
 $
a_k
=
U_k
+
\frac{1}{\eta_k(C_\alpha - 1)} D_k,
\alpha = 1 - \frac{1}{C_\alpha},
b_k = 4B_\alpha  \epsilon,$ leads to
\[
U_k
\le
\left(1 - \frac{1}{C_\alpha}\right)^k
\left(
U_0
+
\frac{1}{\eta_0(C_\alpha - 1)} D_0
\right)
+
4C_\alpha B_\alpha  \epsilon .
\]

\end{proof}

\subsection{Proof of Theorem~\ref{thm:spmg}}
First, for the finiteness of the sample complexity in Theorem~\ref{thm:spmg}, we prove the following lemma.
\begin{lemma}\label{lem:finite}
For  $\alpha \in (0,1/|\cA|)$, 
    $\max_{\pi \in \Pi_\alpha}\infn{Q^{\pi}}< \infty  ,  \max_{\pi \in \Pi_\alpha}t^{\pi}_{\mathrm{tar}}< \infty.$
\end{lemma}
\begin{proof}
As shown in the proof of Lemma~\ref{lem:con}, \(P_\star^\pi\) is continuous in \(\pi \in \Pi_+\), and this implies the continuity of \(V^\pi\) and \(Q^\pi\) by definition. Next, for \(s,s' \in \cR_i\) with $s\ne s'$, 
$
\mathbb{E}_{\pi}\!\left[t^i_{s'} \mid s_0=s \right]
= e_s^\top (I-R^\pi_{i,s'})^{-1}\mathbf{1},
$
where $e_s$ is the coordinate vector indexed by $\cR_i\setminus\{s'\}$ and \(R^\pi_{i,s'}\) denotes the submatrix of \(R_i^\pi\) obtained by removing the state \(s'\) from \(\cR_i\), and the absolute values of its eigenvalues are all smaller than 1. For $s=s'$, the expected hitting time is $0$. Then, continuity of the stationary distribution of \(R_i^\pi\) and \(P_\star^\pi\) implies the continuity of \(t_{\mathrm{tar}}^\pi\) by definition. Finally, since \(\Pi_\alpha\) is compact, the desired result follows.
\end{proof}
Now, we are ready to prove Theorem~\ref{thm:spmg}.

\begin{proof}
First, consider the constant step size and the corresponding integer $K$ in Theorem~\ref{thm:spmg}. Apply Theorem~\ref{thm:PE} at each iteration with failure probability $\delta/K$ and critic error
\[
\infn{G^{\pi_k}-\widehat G^{\pi_k}}\le\frac{\epsilon}{2(2C_\alpha+K+2)B_\alpha}.
\]
By a union bound, this holds simultaneously for $0\le k\le K-1$ with probability at least $1-\delta$. Since $0<\epsilon<1$,
\[
K\le\frac{2(1+A_\alpha)}{\epsilon},\qquad
2C_\alpha+K+2\le\frac{3(1+C_\alpha+A_\alpha)}{\epsilon}.
\]
Substituting these bounds into the total sample complexity gives
\[
\mathcal{O}\left(\frac{(1+R)^3(1+t_\alpha)^3(1+Q_\alpha)^4 B_\alpha^6 (1+A_\alpha)(1+C_\alpha+A_\alpha)^6|\cS||\cA|}{\epsilon^{13}}\log\left(\frac{4(1+A_\alpha)|\cS||\cA|}{\delta\epsilon}\right)\right).
\]
Theorem~\ref{thm:inexact} then gives $J_\mu^{\pi_\alpha}-J_\mu^{\pi_K}\le\epsilon/2+\epsilon/2=\epsilon$.

For the adaptive step sizes and the corresponding integer $K$ in Theorem~\ref{thm:spmg}, apply Theorem~\ref{thm:PE} with failure probability $\delta/K$ and critic error $\epsilon/(8C_\alpha B_\alpha)$. Define
\[
A'_\alpha:=J_\mu^{\pi_\alpha}-J_\mu^{\pi_0}
+\frac{D_{\rho_\mu^{\pi_\alpha}}(\pi_\alpha,\pi_0)}{\eta_0(C_\alpha-1)}.
\]
Since $\log(C_\alpha/(C_\alpha-1))\ge 1/C_\alpha$, we have
\[
K\le 1+C_\alpha\log\max\{1,2A'_\alpha/\epsilon\}
\le C_\alpha\bigl(1+\log\max\{1,2A'_\alpha/\epsilon\}\bigr).
\]
Substituting this bound into the total sample complexity and suppressing logarithmic factors gives
\[
\widetilde{\mathcal O}\left(\frac{(1+R)^3(1+t_\alpha)^3(1+Q_\alpha)^4 C_\alpha^7 B_\alpha^6 |\cS||\cA|}{\epsilon^6}\right).
\]
The union bound and Theorem~\ref{thm:inexact} again give $J_\mu^{\pi_\alpha}-J_\mu^{\pi_K}\le\epsilon/2+\epsilon/2=\epsilon$ with probability at least $1-\delta$.
\end{proof}

\section{Weakly communicating MDPs}\label{weak_mdp}
In this section, we assume that MDP is weakly communicating. We will show that $\alpha$-clipped policy mirror ascent algorithm attains an $\epsilon$-optimal policy satisfying $J^\star_\mu-J^\pi_\mu \le \epsilon$ (i.e., not restricted to positive policies) by choosing $\alpha$ as a function of $\epsilon$. 

 \subsection{Continuity of $J^\pi$ in weakly communicating}
 We first argue about continuity of $J^\pi$ at optimal policy in weakly communicating MDP.
 \begin{lemma}\label{lem:weak}
 In weakly communicating MDP, the mappings $\pi \mapsto J^\pi$ and $\pi \mapsto J^\pi_\mu$ are continuous at optimal policies, and $J^\star_{+,\mu}= J^\star_\mu$ for any fixed $\mu \in \mathcal{M}(\mathcal{S})$. 
\end{lemma}
\begin{proof}
For weakly communicating MDP, optimal average reward is uniform vector \cite[Theorem 8.32]{10.5555/528623}. Then, in the proof of Lemma~\ref{lem:pdl}, we have $(P^{\pi_\star}_\star-I) J^{\star}=0$, and  we have
\begin{align*}
J_\mu^{\pi} - J_\mu^{\star}
&=    \sum_{s \in \cR}
\sum_{a \in \cA}
d^{\pi}_{\mu}(s) \big(\pi(a \mid s)-\pi_\star(a \mid s)\big)Q^{\pi_\star}(s,a). 
\end{align*}
By  Holder's inequality,   $J_\mu^{\star}-J_\mu^{\pi} \le \sum_{s \in \cR} d^\pi_\mu(s)
 \big \|\pi(\cdot \mid s)-\pi_\star(\cdot \mid s)\big\|_1\infn{Q^{\pi_\star}} $. So $\lim_{\pi \rightarrow \pi_\star} J^\pi_\mu= J^\star_\mu$ for $\pi\in \Pi_+$.

\end{proof}

 Given $\mu \in \cM(\cS)$, we say $\pi$ is an $\epsilon$-optimal policy if $J_\mu^{\star}-J_\mu^{\pi} \le \epsilon$. Then, Lemma~\ref{lem:weak} guarantees that we can find $\epsilon$-optimal policy for arbitrary $\epsilon>0$ by only considering policies in $\Pi_+$.

\subsection{Policy gradient theorem in weakly communicating MDP}

We first state following key lemma for weakly communicating  MDPs.
\begin{lemma}\label{lemm:weak}
Consider a weakly communicating MDP. For every $\pi \in \Pi_+$, $J^\pi =c^\pi \mathbf{1}$ for some $c^\pi \in \real$.  
\end{lemma}

\begin{proof}
Let $\mathcal R$ be the closed communicating set in the definition of a weakly communicating MDP. Its union transition graph, containing every edge with positive probability under some action, is strongly connected. Every $\pi\in\Pi_+$ assigns positive probability to every such edge, so $\mathcal R$ is a single irreducible recurrent class under $\pi$. All remaining states are transient under every policy by definition. Hence $P^\pi_\star=\mathbf{1}(g^\pi)^\top$, where $g^\pi$ is the unique stationary distribution, and $J^\pi=((g^\pi)^\top r^\pi)\mathbf{1}$.
\end{proof}
Applying Lemma~\ref{lemm:weak} into the proof of Lemma \ref{lem:pdl}, we can directly obtain following performance difference lemma for weakly communicating MDP. 
\begin{corollary}
Consider a weakly communicating MDP. For $\pi, \pi' \in \Pi_+$ and $\mu\in \mathcal{M}_+(\mathcal{S})$, 
\begin{align*}
J_\mu^{\pi} - J_\mu^{\pi'}
&=    \sum_{s \in \cR}
\sum_{a \in \cA}
d^{\pi}_{\mu}(s) \big(\pi(a \mid s)-\pi'(a \mid s)\big)Q^{\pi'}(s,a). 
\end{align*}
\end{corollary}
Again, applying Lemma~\ref{lemm:weak} into the proof of Theorem \ref{thm::policy_gd}, we can directly obtain following policy gradient theorem for weakly communicating MDP.

\begin{corollary}
 Consider a weakly communicating MDP. 
 For $\pi_{\theta} \in \Pi_{+}$  and $\mu\in \mathcal{M}_+(\mathcal{S})$,
\[\nabla_{\theta}  J_\mu^{\pi_\theta} =    \sum_{s \in \cR}
\sum_{a \in \cA}
d^{\pi_\theta}_{\mu}(s) \nabla_{\theta} \pi_{\theta} (a \mid s) Q^{\pi_\theta}(s,a). \]
\end{corollary}

\subsection{Convergence of policy mirror ascent in weakly communicating MDPs}
If an MDP is weakly communicating, we can explicitly define an iteration number for finding an $\epsilon$-optimal policy with the value of $\alpha$ chosen to be a function of $\epsilon$. 

For the choices of $\alpha$ below, assume $|\cA|\ge2$ and $\infn{Q^{\pi^\star}}>0$. If $|\cA|=1$ or $\infn{Q^{\pi^\star}}=0$, every policy is optimal, and one may set $\alpha=1/(2|\cA|)$ and return $\pi_0$.

We first present the \emph{sublinear} convergence rate of the $\alpha$-clipped policy mirror ascent with constant step sizes. 

\begin{corollary}\label{cor:proj_grad}
Consider weakly communicating MDP. For any given $\epsilon \in (0,1)$,
set $\alpha = \min\left\{\frac{1}{2|\cA|},\frac{\epsilon}{4(|\cA|-1)\infn{Q^{\pi^\star}}}\right\}  $.
Then, for $\pi_0 \in \Pi_\alpha$ and $\mu \in \cM_+(\cS)$, $\alpha$-clipped policy mirror ascent with constant step size $\eta$ generates $\epsilon$-optimal policies for $k\ge \frac{2}{\epsilon}\left(\frac{D_{\rho_\mu^{\pi_\alpha}}(\pi_\alpha,\pi_0)}{\eta}
+C_\alpha  (J_\mu^{\pi_\alpha}-J_\mu^{\pi_0})\right).$
\end{corollary}

\begin{proof}
Define $\bar\pi_\alpha(a\mid s)=\alpha+(1-|\cA|\alpha)\pi^\star(a\mid s)$. Then $\bar\pi_\alpha\in\Pi_\alpha$,
\[
\infn{\bar\pi_\alpha-\pi^\star}\le(|\cA|-1)\alpha,
\qquad\max_s\|\bar\pi_\alpha(\cdot\mid s)-\pi^\star(\cdot\mid s)\|_1\le2(|\cA|-1)\alpha.
\]
Following the proof of Lemma~\ref{lem:weak},
\[
J_\mu^\star-J_\mu^{\pi_\alpha}
\le J_\mu^\star-J_\mu^{\bar\pi_\alpha}
\le2(|\cA|-1)\alpha\infn{Q^{\pi^\star}}
\le\frac{\epsilon}{2}.
\]
By Theorem~\ref{thm:ngd}, when $k\ge\frac{2}{\epsilon}\left(\frac{D_{\rho_\mu^{\pi_\alpha}}(\pi_\alpha,\pi_0)}{\eta}+C_\alpha(J_\mu^{\pi_\alpha}-J_\mu^{\pi_0})\right)$, we have $J_\mu^{\pi_\alpha}-J_\mu^{\pi_k}\le\epsilon/2$. Adding the two bounds concludes the proof.
\end{proof}

Next, we present the \emph{linear} convergence rate of the $\alpha$-clipped policy mirror ascent with adaptive step sizes. 

\begin{corollary}\label{cor:ngd_linear}
Consider a weakly communicating MDP.  For any given $\epsilon \in (0,1)$, 
set $\alpha = \min\left\{\frac{1}{2|\cA|},\frac{\epsilon}{4(|\cA|-1)\infn{Q^{\pi^\star}}}\right\}  $.
Then, for $\pi_0 \in \Pi_\alpha$ and $\mu\in\cM_+(\cS) $, $\alpha$-clipped policy mirror ascent with step sizes $(C_\alpha-1)\eta_{k+1}\ge C_\alpha\eta_k>0$ generates $\epsilon$-optimal policies for  
\[
k\ge \frac{
\log\max\left\{1,2(J_\mu^{\pi_\alpha}- J_\mu^{\pi_{0}}
+ \frac{D_{\rho_\mu^{\pi_\alpha}}(\pi_\alpha, \pi_{0})}{\eta_0(C_\alpha-1)})/\epsilon
\right\}}{ \log(C_\alpha /  (C_\alpha-1)) }
\].
\end{corollary}

\begin{proof}
The construction in the proof of Corollary~\ref{cor:proj_grad} gives $J_\mu^\star-J_\mu^{\pi_\alpha}\le\epsilon/2$ for the stated choice of $\alpha$.
By Theorem~\ref{thm:ngd_linear}, the iteration bound in this corollary gives $J_\mu^{\pi_\alpha}-J_\mu^{\pi_k}\le\epsilon/2$. Adding the two bounds concludes the proof.
\end{proof}
To the best of our knowledge, Corollaries~\ref{cor:proj_grad} and \ref{cor:ngd_linear} are the first convergence results of policy gradient methods for average-reward weakly communicating MDPs.
\subsection{Sample complexity of policy mirror ascent in weakly communicating MDPs}
If an MDP is weakly communicating, we can explicitly provide sample complexity for finding an $\epsilon$-optimal policy. Note that  for $\pi \in \Pi_+$, $\|d^\pi_{\mu}\|_1=1$, but $B_\alpha=\max_{\pi\in\Pi_\alpha}\|\rho_\mu^\pi\|_1$ also includes transient visitation and need not equal $1$.

\begin{corollary}\label{sam:weak}
    Consider a weakly communicating MDP. Let $\epsilon\in(0,1), \delta\in(0,1)$ and set $\alpha = \min\left\{\frac{1}{2|\cA|},\frac{\epsilon}{4(|\cA|-1)\infn{Q^{\pi^\star}}}\right\}  $. For any $\pi, \pi_0 \in \Pi_\alpha$ and given $\mu \in \cM_+(\cS)$, with probability $1-\delta$, the iterate of stochastic $\alpha$-clipped policy ascent with constant step size $\eta>0$ and $K=\max\left\{1,\left\lceil \frac{4}{\epsilon}\left(\frac{D_{\rho_\mu^{\pi_\alpha}}(\pi_\alpha,\pi_0)}{\eta}+C_\alpha(J_\mu^{\pi_\alpha}-J_\mu^{\pi_0})\right)\right\rceil\right\}$ generates $\epsilon$-optimal policy with sample complexity 
\[\widetilde{\mathcal{O}}\left(\frac{(1+R)^3(1+t_\alpha)^3(1+Q_\alpha)^4 B_\alpha^6 K(2C_\alpha+K+2)^6|\cS||\cA|}{\epsilon^6}  \right)\]
and with adaptive step size $\eta_{k+1}(C_\alpha  - 1) \ge \eta_k C_\alpha >0$ and $K= \max\left\{1,\left\lceil\frac{\log\max\left\{1,\frac{4}{\epsilon}\left(J_\mu^{\pi_\alpha}-J_\mu^{\pi_0}+\frac{D_{\rho_\mu^{\pi_\alpha}}(\pi_\alpha,\pi_0)}{\eta_0(C_\alpha-1)}\right)\right\}}{\log(C_\alpha/(C_\alpha-1))}\right\rceil\right\}$ generates $\epsilon$-optimal policy with sample complexity 
\[\tilde{O}\left( \frac{(1+R)^3(1+t_\alpha)^3(1+Q_\alpha)^4 C_\alpha^6 B_\alpha^6 K|\cS||\cA|}{ \epsilon^6}\right)\]
where $\widetilde{\mathcal{O}}$ ignores all logarithmic factors, $Q_{\alpha}= \max_{\pi \in \Pi_\alpha}\infn{Q^{\pi}}$,  and $t_\alpha=\max_{\pi \in \Pi_\alpha}t^{\pi}_{\mathrm{tar}}$.
\end{corollary}

\begin{proof}
The construction in the proof of Corollary~\ref{cor:proj_grad} gives $J_\mu^\star-J_\mu^{\pi_\alpha}\le\epsilon/2$.
Apply Theorem~\ref{thm:spmg} with accuracy $\epsilon/2$ and confidence parameter $\delta$. The stated integer $K$ then gives $J_\mu^{\pi_\alpha}-J_\mu^{\pi_K}\le\epsilon/2$ with probability at least $1-\delta$.
For constant step size, the total sample complexity is
\[
\mathcal{O}\left(\frac{(1+R)^3(1+t_\alpha)^3(1+Q_\alpha)^4 B_\alpha^6 K(2C_\alpha+K+2)^6|\cS||\cA|}{\epsilon^6}\log\left(\frac{2K|\cS||\cA|}{\delta}\right)\right).
\]
For adaptive step sizes, it is
\[
\mathcal{O}\left(\frac{(1+R)^3(1+t_\alpha)^3(1+Q_\alpha)^4 C_\alpha^6 B_\alpha^6 K|\cS||\cA|}{\epsilon^6}\log\left(\frac{2K|\cS||\cA|}{\delta}\right)\right).
\]
In either case, $J_\mu^\star-J_\mu^{\pi_K}\le\epsilon/2+\epsilon/2=\epsilon$.
\end{proof}

To the best of our knowledge, Corollary~\ref{sam:weak} is the first sample-complexity result of policy gradient methods for average-reward weakly communicating MDPs.

\section{Projection algorithms}\label{appen:proj}

\begin{algorithm}[H]
\caption{Euclidean projection onto $\cM_\alpha (\cA)$}
\label{alg:euclid_K_floor}
\begin{algorithmic}
\STATE \textbf{Input:} $q\in\mathbb R^d$, $\alpha\in[0,1/d]$
\IF{$\alpha=1/d$}
    \STATE \textbf{return} $p=(1/d,\ldots,1/d)$
\ENDIF
\STATE Sort $q$ into $q'$ satisfying $q'_1 \ge q'_2 \ge \cdots \ge q'_d$
\STATE Find $
\rho = \max\Bigl\{1 \le j \le d:\; q'_j + \frac{1}{j}\Bigl(1-d\alpha -\sum_{i=1}^{j} q'_i\Bigr) > 0 \Bigr\}$
\STATE $\lambda=\frac{1}{\rho}(1-d\alpha -\sum_{i=1}^{\rho} q'_i\Bigr) $ 
\FOR{$i=1,\ldots,d$}
    \STATE  $p_i=\max\{q_i+\lambda+\alpha,\alpha\}$
\ENDFOR
\STATE \textbf{Output:} $p$
\end{algorithmic}
\end{algorithm}

\vspace{-0.1in}
\begin{algorithm}[H]
\caption{KL projection onto $\cM_\alpha (\cA)$}
\label{alg:kl_proj_lower_bounded_simplex_median}
\begin{algorithmic}
\STATE \textbf{Input:}  $w\in \cM_+(\cS)$  where $|\cS|=d$, $\alpha\in[0,1/d]$
\STATE Set $W = \{1,\ldots,d\}$, $C_{\#}= 0$,  $C_{\%} = 0$
\WHILE{$W \neq \emptyset$}
    \STATE Find median $\omega$ in $\{w_i :i\in W \}$
    \STATE $L = \{i\in W : w_i < \omega\}$,\quad
           $L_{\#} = |L|$,\quad
           $L_{\%} \gets \sum_{i\in L} w_i$
    \STATE $M =\{i\in W : w_i = \omega\}$,\quad
           $M_{\#} =|M|$,\quad
           $M_{\%} =\sum_{i\in M} w_i$
    \STATE $H =\{i\in W : w_i > \omega\}$
    \STATE $m_0 =\dfrac{1-(C_{\#}+L_{\#})\alpha}{1-(C_{\%}+L_{\%})}$
    \IF{$\omega\, m_0 < \alpha$}
        \STATE $C_{\#} =C_{\#}+L_{\#}+M_{\#}$
        \STATE $C_{\%} =C_{\%}+L_{\%}+M_{\%}$
        \STATE $W =H$
    \ELSE
        \STATE $W =L$
    \ENDIF
\ENDWHILE
\STATE $m_0 =\dfrac{1-C_{\#}\alpha}{1-C_{\%}}$
\FOR{$i=1,\ldots,d$}
    \STATE $p_i =\max\{\alpha,m_0\,w_i\}$
\ENDFOR
\STATE \textbf{Output:} $p$
\end{algorithmic}
\end{algorithm}


\newpage
\section*{NeurIPS Paper Checklist}

\begin{enumerate}

\item {\bf Claims}
    \item[] Question: Do the main claims made in the abstract and introduction accurately reflect the paper's contributions and scope?
    \item[] Answer: \answerYes{} 
    \item[] Justification: Yes main claims made in the abstract and introduction accurately reflect the paper's contributions and scope.  

\item {\bf Limitations}
    \item[] Question: Does the paper discuss the limitations of the work performed by the authors?
    \item[] Answer: \answerYes{} 
    \item[] Justification: Yes, authors  discuss the limitations of the work in conclusion section.

\item {\bf Theory assumptions and proofs}
    \item[] Question: For each theoretical result, does the paper provide the full set of assumptions and a complete (and correct) proof?
    \item[] Answer: \answerYes{} 
    \item[] Justification: Yes,  the paper provide the full set of assumptions and a complete (and correct) proof.

    \item {\bf Experimental result reproducibility}
    \item[] Question: Does the paper fully disclose all the information needed to reproduce the main experimental results of the paper to the extent that it affects the main claims and/or conclusions of the paper (regardless of whether the code and data are provided or not)?
    \item[] Answer: \answerNA{} 
    \item[] Justification: This paper does not include the  experimental results.

\item {\bf Open access to data and code}
    \item[] Question: Does the paper provide open access to the data and code, with sufficient instructions to faithfully reproduce the main experimental results, as described in supplemental material?
    \item[] Answer: \answerNA{} 
    \item[] Justification: This paper does not include any experimental result.

\item {\bf Experimental setting/details}
    \item[] Question: Does the paper specify all the training and test details (e.g., data splits, hyperparameters, how they were chosen, type of optimizer) necessary to understand the results?
    \item[] Answer:  \answerNA{} 
    \item[] Justification: This paper does not include any experimental result.

\item {\bf Experiment statistical significance}
    \item[] Question: Does the paper report error bars suitably and correctly defined or other appropriate information about the statistical significance of the experiments?
    \item[] Answer: \answerNA{} 
    \item[] Justification: This paper does not include any experiment.

\item {\bf Experiments compute resources}
    \item[] Question: For each experiment, does the paper provide sufficient information on the computer resources (type of compute workers, memory, time of execution) needed to reproduce the experiments?
    \item[] Answer:  \answerNA{} 
    \item[] Justification: This paper does not include any expereiments.
    
\item {\bf Code of ethics}
    \item[] Question: Does the research conducted in the paper conform, in every respect, with the NeurIPS Code of Ethics \url{https://neurips.cc/public/EthicsGuidelines}?
    \item[] Answer: \answerYes{} 
    \item[] Justification: This research conducted in the paper conform, in every respect, with the NeurIPS Code of Ethics.

\item {\bf Broader impacts}
    \item[] Question: Does the paper discuss both potential positive societal impacts and negative societal impacts of the work performed?
    \item[] Answer: \answerNA{} 
    \item[] Justification: This paper is theory paper and has no potential negative or positive societal impacts. 
    
\item {\bf Safeguards}
    \item[] Question: Does the paper describe safeguards that have been put in place for responsible release of data or models that have a high risk for misuse (e.g., pre-trained language models, image generators, or scraped datasets)?
    \item[] Answer: \answerNA{} 
    \item[] Justification: This paper does not include any experiments.

\item {\bf Licenses for existing assets}
    \item[] Question: Are the creators or original owners of assets (e.g., code, data, models), used in the paper, properly credited and are the license and terms of use explicitly mentioned and properly respected?
    \item[] Answer: \answerNA{}
    \item[] Justification: This paper does not include any experiments.

\item {\bf New assets}
    \item[] Question: Are new assets introduced in the paper well documented and is the documentation provided alongside the assets?
    \item[] Answer: \answerNA{}
    \item[] Justification: This paper does not include any experiments.

\item {\bf Crowdsourcing and research with human subjects}
    \item[] Question: For crowdsourcing experiments and research with human subjects, does the paper include the full text of instructions given to participants and screenshots, if applicable, as well as details about compensation (if any)? 
    \item[] Answer: \answerNA{} 
    \item[] Justification: This paper does not include any experiments.

\item {\bf Institutional review board (IRB) approvals or equivalent for research with human subjects}
    \item[] Question: Does the paper describe potential risks incurred by study participants, whether such risks were disclosed to the subjects, and whether Institutional Review Board (IRB) approvals (or an equivalent approval/review based on the requirements of your country or institution) were obtained?
    \item[] Answer: \answerNA{} 
    \item[] Justification: There is no potential risks incurred by study participants, whether
such risks were disclosed to the subjects, and whether Institutional Review Board (IRB)
approvals 
with human subjects.

\item {\bf Declaration of LLM usage}
    \item[] Question: Does the paper describe the usage of LLMs if it is an important, original, or non-standard component of the core methods in this research? Note that if the LLM is used only for writing, editing, or formatting purposes and does \emph{not} impact the core methodology, scientific rigor, or originality of the research, declaration is not required.
    \item[] Answer: \answerNA{} 
    \item[] Justification: We use LLM to fix grammar error.

\end{enumerate}

\end{document}